\documentclass[a4paper, english,reqno]{amsart}
\usepackage{amsmath} 
\usepackage{amsthm} 
\usepackage{amssymb} 
\usepackage{amscd} 

\usepackage{array} 
\usepackage[backref=page,linktocpage]{hyperref} 
\usepackage{caption} %
\usepackage{graphics,graphicx} 
\usepackage{tikz,tikz-cd} 

\usepackage{enumerate} 

\usepackage{xcolor}

\makeatletter
\@namedef{subjclassname@2020}{\textup{2020} Mathematics Subject Classification}
\makeatother

\hypersetup{
    colorlinks = true,
    linkbordercolor = {white},
    linkcolor = {blue},
    anchorcolor = {black},
    citecolor = {blue},
    filecolor = {cyan},
    menucolor = {red},
    runcolor = {cyan},
    urlcolor = {black}
}

\usetikzlibrary{automata}

\numberwithin{equation}{section}
\theoremstyle{plain}
   \newtheorem{theorem}{Theorem}[section]
   \newtheorem{proposition}[theorem]{Proposition}
   \newtheorem{lemma}[theorem]{Lemma}
   \newtheorem{corollary}[theorem]{Corollary}

\theoremstyle{definition}
   \newtheorem{definition}[theorem]{Definition}
   \newtheorem{example}[theorem]{Example}

\theoremstyle{remark}
   \newtheorem{remark}[theorem]{Remark}

\DeclareMathOperator{\TN}{\mathrm{TN}}

\newcommand{\A}{\mathrm{A}}

\newcommand{\zero}{\hat{0}}
\newcommand{\one}{\hat{1}}
\newcommand{\NN}{\mathbb{N}}
\newcommand{\ZZ}{\mathbb{Z}}
\newcommand{\RR}{\mathbb{R}}
\newcommand{\R}{\mathrm{R}}
\newcommand{\CC}{\mathbb{C}}
\newcommand{\BB}{\mathbb{B}}
\newcommand{\DD}{\mathcal{D}}
\renewcommand{\A}{\mathcal{A}}
\newcommand{\SSS}{\mathcal{S}}
\newcommand{\I}{\mathcal{I}}
\newcommand{\xx}{\mathbf{x}}
\newcommand{\yy}{\mathbf{y}}
\newcommand{\des}{\mathrm{des}}

\newcommand{\col}{\mathrm{col}}
\newcommand{\aug}{\mathrm{aug}}

\title{Chow polynomials of totally nonnegative matrices and posets}
\author{Petter Br\"and\'en}
\author{Lorenzo Vecchi}
\address{Department of Mathematics, KTH Royal Institute of Technology}

\keywords{Totally nonnegative matrix, real-rooted polynomial, P\'olya frequency sequence, Chow polynomial, Chow ring, geometric lattice, binomial poset, Sheffer poset}
\subjclass[2020]{06A07, 15B48, 13D40, 26C10}

\thanks{PB is a Wallenberg Academy Scholar
  supported by the Knut and Alice Wallenberg
  Foundation, and the G\"oran Gustafsson foundation.}

\begin{document}

\begin{abstract}
   Huh-Stevens and Ferroni-Schr\"oter independently conjectured that Hilbert-Poincaré series of Chow rings of  geometric lattices have only real zeros. Ferroni, Matherne and the second author extended this conjecture to Chow polynomials of Cohen-Macaulay poset. In this paper we address the above conjectures by providing new defining relations and properties of Chow functions of posets and matrices. 
    These are used, in conjunction with new techniques on interlacing sequences of polynomials, to prove that Chow polynomials of totally nonnegative matrices have only real zeros, which, in turn, proves the above conjectures for a class of posets  that contains projective and affine geometries, face lattices of cubical polytopes, partition lattices and Dowling lattices, perfect matroid designs, and lattices of flats of paving matroids.  

    We also study Chow polynomials of Toeplitz matrices in greater detail, and show how these are related the combinatorics of binomial and Sheffer posets, as well as to a family of generalized Eulerian polynomials with coefficients in the ring of symmetric polynomials that have been studied  by e.g. Stanley, Brenti, Stembridge and Shareshian-Wachs.
\end{abstract}

\thispagestyle{empty}
\maketitle
\setcounter{tocdepth}{2}
\tableofcontents
\thispagestyle{empty}
\newpage
\section{Introduction}
The study of algebraic invariants of posets and matroids has seen a remarkable development over the past two decades. It has led not only to the solution of long-standing conjectures on positivity in matroid theory, see \cite{eur}, but also opened up new lines of research and led to new positivity questions. For example, the Kazhdan-Lusztig polynomial, the $Z$-polynomial, the Hilbert-Poincaré series of the Chow ring, and the chain polynomial of the lattice of flats of any matroid have been conjectured to have only real zeros \cite{gedeon-proudfoot-young-survey, proudfoot-xu-young, stevens-bachelor, ferroni-schroter, athanasiadis-kalampogia}. To match the complex recursions
satisfied by these polynomials, there is a need for new methods on the geometry of zeros of polynomials that can handle such recursions. Equally important is to reveal new relations and
recursions among these polynomials, ones that match better the type of recursions
that the existing theory on zeros of polynomials can handle. These are the two motivating purposes of the paper.

In their influential work \cite{feichtner-yuzvinsky}, Feichtner and Yuzvinsky introduced the Chow ring of an atomistic lattice $L$, and gave a presentation for it via relations in terms of chains in the lattice. This ring is graded and finite-dimensional, allowing one to study its Hilbert-Poincaré series as an invariant of the lattice. When $L$ is a geometric lattice, then the Chow ring satisfies the K\"ahler package \cite{adiprasito-huh-katz}, which implies that the Hilbert series of the Chow ring of a geometric lattice  is always palindromic and unimodal. Moreover for specific classes of geometric lattices, these Hilbert series coincide with well-known families of polynomials; for example  for boolean algebras one recovers the Eulerian polynomials, and for truncations of boolean algebras one recovers the derangement polynomials \cite{hameister-rao-simpson}. For these reasons, the Hilbert series of Chow rings of geometric lattices (or matroids)  have garnered importance on their own and are now known as the \emph{Chow polynomials} of matroids. A tightly connected construction is the one of augmented Chow ring of $L$ introduced in \cite{semismall}, whose Hilbert series is known as the \emph{augmented Chow polynomial}.

Ferroni-Schr\"oter \cite[Conjecture 8.18]{ferroni-schroter} and Huh-Stevens  \cite[Conjectures 4.1.3 and 4.3.3]{stevens-bachelor} independently   conjectured that Chow polynomials and augmented Chow polynomials of matroids have only real zeros. This was recently verified in the uniform case in  \cite[Theorem 1.1]{branden-vecchi} and  \cite[Theorem 1.10]{ferroni-matherne-stevens-vecchi}, respectively. 
Ferroni, Matherne and the second author extended the definition of Chow polynomials to arbitrary bounded posets using the framework of Kazhdan-Lusztig-Stanley theory \cite{ferroni-matherne-vecchi}. These are palindromic polynomials that are conjectured to have only real zeros for all Cohen-Macaulay posets and Bruhat intervals (with the appropriate $P$-kernels \cite[Conjectures 4.26 and 6.14]{ferroni-matherne-vecchi}). In this larger context, Hoster and Stump \cite[Theorem 1.1]{hosterstump} proved that Chow polynomials of boolean complexes with nonnegative $h$-vectors are real-rooted.

In this paper we develop a systematic approach to the  study of zeros of Chow polynomials for a class of posets called $\TN$-\emph{posets}, which were introduced in \cite{branden-saud-1,branden-saud-2}. This class includes (rank-selected subposets of) affine and projective geometries, perfect matroid designs, dual partition lattices and dual Dowling lattices. To do so, we define and study Chow polynomials of \emph{totally nonnegative  matrices} ($\TN$-matrices), i.e., matrices whose minors are all nonnegative. Our main result is Theorem \ref{mainman2}, which states that the Chow polynomial of any lower triangular $\TN$-matrix with all diagonal entries equal to one is real-rooted. A direct consequence of this is Theorem~\ref{mainposet}, which says that the Chow polynomial of any $\TN$-poset (or dual of any $\TN$-poset) is real-rooted. We also use our methods to prove that the Chow polynomial of any paving matroid is real-rooted, Corollary~\ref{cor:main paving}. \\

\noindent 
{\bf Outline.} The structure of the paper is as follows. 
In Section \ref{section: Chow posets} we review relevant Kazhdan–Lusztig–Stanley theory for posets, and provide new defining relations and properties of Chow functions (Theorem \ref{thm: def H and d}), and augmented Chow functions (Theorem \ref{thm: def G and A}). These relations reveal that a Chow polynomial always comes paired with another polynomial that we call a Chow-derangement polynomial. 
This pairing may be seen as a  generalization of the classical pairing of Eulerian polynomials and derangement polynomials, see e.g. \cite{Brenti-90, branden-solus}. 

 In Section \ref{section chow matrix} we use the machinery developed in Section \ref{section: Chow posets} to define and study (augmented) Chow polynomials of matrices. We relate this construction to (augmented) Chow polynomials of weak-rank uniform posets, a class of highly structured posets. We define and study a notion of $\I_n$-interlacing sequences of polynomials, which refines the notion of interlacing sequences to the case when the zeros of the polynomial in question interlace the zeros of its reciprocal. The new defining relations for Chow functions obtained in Section \ref{section: Chow posets} translate into recursions that preserve $\I_n$-interlacing sequences of polynomials  when combined with resolvability of $\TN$-matrices (Definition \ref{resolv}), thus refining the method using resolvability that was initiated in \cite{branden-saud-1}.  This is used to prove interlacing preserving properties of the so called Chow-deranged map and the Chow-Eulerian transformation associated to a matrix $R$. These are vast generalizations of two results of the first author and Solus \cite{branden-solus}, and Athanasiadis \cite{athanasiadis-eulerian2}, respectively, on real-rootedness preserving properties of the deranged map and the Eulerian transformation. Indeed Theorems \ref{mainman} and \ref{mainmanA} generalize these results to any lower triangular $\TN$-matrix with all diagonal entries equal to one. A direct consequence is Theorem~\ref{mainman2} which states that the Chow polynomials associated to  any such matrix are real-rooted.
 
 In Section \ref{section gamma} we provide an explicit interpretation of the coefficients of the $\gamma$-polynomial of Chow polynomials of matrices as sums of certain minors of the matrix (Corollary \ref{cor: coefficients of gamma as determinants}) and as enumerators of certain non-intersecting paths in a network, following the work by Gessel and Viennot \cite{gessel-viennot}. In Sections~\ref{section real zeros TN posets} and  \ref{pavingsection} we prove Theorem \ref{mainposet} and Theorem \ref{thm:main paving}, which assert that (augmented) Chow polynomials of $\TN$-posets, dual $\TN$-posets and generalized paving posets (a class of posets that contains lattices of flats of paving matroids) are real-rooted. 
 
In Section \ref{section toeplitz}, we study the case of Toeplitz matrices in greater detail, and provide generating function identities for the Chow polynomials of interest in Theorem~\ref{fourform}. In Theorem~\ref{PFSR}, we prove that the Chow polynomials associated to Toeplitz matrices of P\'olya frequency sequences are real-rooted. In  Section \ref{binopo} we provide expressions for the generating series for the various Chow polynomials in the case when  $P$ is a binomial poset or a Sheffer poset. These are classes of  highly regular posets studied by Doubilet, Rota and Stanley \cite{DRS}, and Ehrenborg and Readdy \cite{Ehr-Rea}, respectively. 

 In Section \ref{symsec} we show that the Chow polynomials associated to the Toeplitz matrices  $(e_{i-j}(\xx))_{i,j=0}^\infty$  and $(h_{i-j}(\xx))_{i,j=0}^\infty$ coincide with the generalized Eulerian polynomials, with symmetric function coefficients, that have been studied frequently in the literature, by Stanley, Brenti, Stembridge, Shareshian and Wachs and others, see \cite{Shareshian-Wachs-10} and the references therein. Hence the theory developed in this paper serves as a framework for these polynomials, that, for example, enables us to reprove Schur $\gamma$-positivity results of Gessel and Shareshian-Wachs, as well as real-rootedness for the polynomials obtained when $\xx$ and $\yy$ are chosen to be suitable real and nonnegative vectors.

\section{The incidence algebra and Chow functions}\label{section: Chow posets}
In this section we recall the construction of Chow functions and Chow polynomials of partially ordered sets (posets). For undefined poset terminology, we refer to \cite{stanley-ec1}. Recall that an interval of a poset is a subposet of $P$ of the form $[x,y]= \{z \in P : x\leq z \leq y\}$, where $x,y\in P$. 
All posets $P$ considered in this paper are \emph{locally finite}, i.e., each interval of $P$ is finite. We also say that a poset is \emph{bounded} if it has unique least and largest elements, which we denote by $\zero$ and $\one$, respectively.
Given a poset $P$, a \emph{weak rank function} is a function $\rho: P\times P \to \NN$ such that 
\begin{itemize}
    \item $\rho(x,y) = \rho_{x,y} > 0$ if and only if $x < y$, and
    \item  $\rho_{x,y} = \rho_{x,z} + \rho_{z,y}$, for all  $x\leq z \leq y$ in $P$.
\end{itemize}
A \emph{weakly ranked} poset consists of a pair $(P,\rho)$, where $\rho$ is a weak rank function for $P$. By slight abuse of notation, we say that $P$ is a weakly ranked poset, when this does not create confusion.
If the poset has a least element $\zero$, we write $\rho(x):= \rho_{\zero,x}$ for an element $x \in P$ and call this the \emph{(weak) rank} of $x$ in $P$. If $\rho_{x,y} =1$ whenever $y$ covers $x$, then we say that $P$ is \emph{graded} (or \emph{ranked}). Moreover if $P$ is bounded, then we say that $P$ has (weak) rank $\rho(\one)$. 

 The \emph{incidence algebra}  of a poset $P$ over a ring $R$ is the free $R$-module spanned by the intervals of $P$.
More explicitly, an element $f$ associates to every interval $[x,y]$ of $P$ an element in $R$ which we denote by $f_{x,y}$. The product (\emph{convolution}) is defined by
\[
(fg)_{x,y} = \sum_{x\leq z \leq y}f_{x,z}g_{z,y}.
\]
We denote by $\delta$ the multiplicative identity. In this paper the ring $R$ will be a polynomial ring $R= \R[t]$, where $\R$ is an integral domain\footnote{It is more common to work with the incidence algebra over $\ZZ[t]$, but we need this more general setting in Section \ref{section toeplitz}.}, and we denote by $I(P)=I_{\R[t]}(P)$ the incidence algebra over $\R[t]$. 
For a polynomial $f \in \R[t]$ of degree at most $n$, we define 
\[
\I_n(f) = t^nf(t^{-1})
\]
and say that $f$ is \emph{palindromic with center of symmetry} $n/2$, if ${\I_n(f) = f}$. 
Let $I_\rho(P)$ the subalgebra of $I(P)$ consisting of functions $f$ such that $\deg f_{x,y} \leq \rho_{x,y}$ for all $x,y \in P$, and  define $\I : I_\rho(P) \to I_\rho(P)$ by 
$$
\mathcal I (f)_{x,y} = \mathcal I_{\rho_{x,y}}(f_{x,y}), \quad \mbox{ for all } x\leq y. 
$$
Hence $\I$ is an involution and an automorphism, i.e., $\I^2=\delta$ and  $\I(fg)=\I(f)\I(g)$ for all $f,g \in I_\rho(P)$. 

Given a weakly ranked poset $P$, an element $\kappa$ in  $I_\rho(P)$ is called a \emph{$P$-kernel} if
\begin{itemize}
    \item $\kappa_{x,x} = 1$ for each $x \in P$, and 
    \item $\kappa^{-1} = \mathcal I(\kappa)$.
\end{itemize}

The following theorem gives a characterization of $P$-kernels. 

\begin{theorem}[{\cite[Theorem~6.2]{brenti-kls}}, {\cite[Proposition 2.5]{proudfoot-kls}}]
    Given a $P$-kernel $\kappa$, there exists a unique element $g\in I(P)$ such that
    \begin{equation}\label{condition for kls}
    g_{x,x}=1 \mbox{ for each } x \in P, \ \ \mbox{ and }  \ \ \deg g_{x,y}< {\rho_{x,y}}/{2} \text{ for every $x < y$ in $P$},
    \end{equation}
and $\kappa = g^{-1}\mathcal I (g)$. Conversely, if $g \in I (P)$ satisfies \eqref{condition for kls}, then $\kappa = g^{-1}\I (g)$ is a $P$-kernel.
\end{theorem}
The function $g$ is called the \emph{left Kazhdan-Lusztig-Stanley (KLS) function} associated to $\kappa$. 

Following \cite[Definition 3.2]{ferroni-matherne-vecchi}, we define the \emph{reduced $P$-kernel} to be the function $\overline{\kappa} \in I_\rho(P)$ defined by
\[
\overline{\kappa} = \frac 1 {t-1} (\kappa-t\delta). 
\]
Notice that $\kappa-t\delta$ evaluated at $t=1$ is equal to $g(1)^{-1}g(1)-\delta=0$, and hence $t-1$ divides $\kappa-t\delta$ and thus $\overline{\kappa}$ is an element of $I_\rho(P)$ as claimed. The \emph{Chow function} associated to $\kappa$ (or $g$) is  defined as the element of $I_\rho(P)$, 
\begin{equation}\label{Chow-def}
    (-\overline{\kappa})^{-1},
\end{equation}
i.e., the negative inverse of the reduced $P$-kernel. 

Define a family of operators $\SSS_n$, $n \in \NN$, on polynomials of degree at most $n$ in $\R[t]$ by 
\[
\SSS_n(f )= \frac{\I_n(f) - f}{t-1}.
\]
The following theorem provides an alternative definition of the Chow function.
\begin{theorem}\label{thm: def H and d}
    Let $g \in I(P)$ be an element that satisfies \eqref{condition for kls}. Then there exist unique functions $H$ and $d$ in $I_\rho(P)$ for which
    \begin{itemize}
        \item[(i)] $d_{x,x} = 1$ for each $x \in P$,
        \item[(ii)] $\I(H) = tH + (1-t)\delta$,
        \item[(iii)] $\I(d) = d$, and 
        \item[(iv)] $H = dg$. 
    \end{itemize}
    Moreover, $H$ is the Chow function \eqref{Chow-def} associated to $g$, and the polynomials $d_{x,y}$ satisfy the recursion 
    \begin{equation}\label{d-recu}
d_{x,y} = t\SSS_{\rho_{x,y}-1}\left(\sum_{x \leq z <y}d_{x,z}g_{z,y}\right),  \quad x<y. 
\end{equation}
\end{theorem}
\begin{proof}
For the existence, let $H$ be the Chow function associated to $g$ and let ${d= Hg^{-1}}$, so that (i) and (iv) are trivially satisfied. 
Notice that in the incidence algebra over the field $\R(t)$, 
$$
H = (1-t)(\kappa-t\delta)^{-1}. 
$$
A simple calculation in this incidence algebra  yields
\begin{align*}
\I(H) &= \I((1-t)(\kappa-t\delta)^{-1})= (1-1/t)(\I(\kappa)-t^{-1}\delta )^{-1} \\
&= (1-t)\kappa (\kappa-t\delta)^{-1}= (1-t)\delta+t H, 
\end{align*}
which verifies (ii). Also 
\begin{equation*}
H\kappa = (1-t)(\kappa-t\delta)^{-1} (\kappa-t\delta+t\delta) = (1-t)\delta+tH=\I(H). 
\end{equation*}
Recall $\kappa= g^{-1}\I(g)$. Then 
$$
\I(d)=\I(H)\I(g^{-1})=H\kappa \I(g^{-1})= Hg^{-1}\I(g)\I(g^{-1})= Hg^{-1}=d,  
$$
which verifies (iii) and completes the proof of the existence of $(H,d)$. 

For the uniqueness, we only need to prove that $d$ is unique. The uniqueness of $H$ then follows from $H = dg$. To prove the uniqueness by induction on $\rho_{x,y}$, it remains to prove \eqref{d-recu}. Suppose $\rho_{x,y} > 0$. By (iv), 
\begin{equation}\label{lalal}
H_{x,y} = d_{x,y} + \sum_{x \leq z <y}d_{x,z}g_{z,y}. 
\end{equation}
By (ii) and (iii), $\SSS_{\rho_{x,y}-1}(H_{x,y})=0$ and $\SSS_{\rho_{x,y}-1}(d_{x,y})= -d_{x,y}/t$. Applying $\SSS_{\rho_{x,y}-1}$ to \eqref{lalal} now yields \eqref{d-recu}. 
\end{proof}
We call the unique function $d$ achieved by Theorem \ref{thm: def H and d} the \emph{Chow-derangement function} associated to $g$. If the poset is bounded, we call $H_P = H_{\zero,\one}$ the \emph{Chow polynomial} of $P$ (with respect to $g$) and $d_P= d_{\zero,\one}$ the \emph{Chow-derangement polynomial} of $P$.

Later in the paper, we will need the following slightly weaker version of Theorem~\ref{thm: def H and d}.
\begin{corollary}\label{cor: H and d for initial intervals}
    Let $P$ be a poset with a unique least element $\zero$ and $g \in I(P)$ be a function satisfying \eqref{condition for kls}. Then, there are two  unique  families of polynomials $\{H_x\}_{x\in P}$ and $\{d_x\}_{x\in P}$ in $\R[t]$ that satisfy 
    \begin{itemize}
        \item[(i)] $d_{\zero} = 1$,
        \item[(ii)] $\I_{\rho_x}(H_x) = tH_x$, for each $x>\zero$, 
        \item[(iii)] $\I_{\rho_x}(d_x) = d_x$, for each $x \in P$, and 
        \item[(iv)] $H_y = \sum_{x \leq y}d_xg_{x,y}$ for each $y \in P$.
    \end{itemize}
    Moreover $H_x = H_{\zero,x}$ and $d_x= d_{\zero, x}$ for each $x \in P$, and 
      \begin{equation}\label{d-recu-2}
d_{y} = t\SSS_{\rho(y)-1}\left(\sum_{x<y}d_{x}g_{x,y}\right),  \quad y \neq \zero.
\end{equation}
\end{corollary}
\begin{proof}
The existence follows directly from Theorem \ref{thm: def H and d}. As in the proof of the uniqueness in Theorem \ref{thm: def H and d}, (ii)--(iv) yields \eqref{d-recu-2}, which also proves the uniqueness. 
\end{proof}

In this paper we will be mostly concerned with a special class of functions $g$, where $g_{x,y} \in \R$ for all $x\leq y$ and $g_{x,x} = 1$ for each $x \in P$. We call such functions \emph{scalar}. An example of a scalar function is $\zeta$, the \emph{zeta function} of $P$, which is defined by $\zeta_{x,y} =1$ for all $x \leq y$ in $P$. 

\begin{example}
    When $g = \zeta$, the corresponding $P$-kernel is the \emph{characteristic function} $\chi$, which associates to each interval its characteristic polynomial. Hence, the associated Chow function $H$ was given the name of \emph{characteristic Chow function} in \cite{ferroni-matherne-vecchi}. The polynomials $H_P$ are $\gamma$-positive for Cohen-Macaulay posets \cite[Theorem~1.4]{ferroni-matherne-vecchi}, and are conjectured to be real-rooted for every Cohen-Macaulay poset \cite[Conjecture~1.5]{ferroni-matherne-vecchi}.
\end{example}

\begin{example}\label{H-Eul-ex}
    When $P$ is a boolean algebra of rank $r$ and $g = \zeta$, then all intervals of the same rank in $P$ are isomorphic, from which it follows that $H_{x}= H_{\rho(x)}$ and $d_{x}= d_{\rho(x)}$ for some polynomials $\{H_n\}_{n=0}^r$ and $\{d_n\}_{n=0}^r$. Corollary \ref{cor: H and d for initial intervals} then says that these two families of polynomials are the unique ones that satisfy $H_0=d_0=1$, 
    \[
    H_n = \sum_{k=0}^n\binom{n}{k}d_k, \quad 0 \leq n \leq N, 
    \]
   $\I_n(d_n)= d_n$ and $\I_{n}(H_n)=tH_n$ for $1\leq n \leq N$. These properties are known to hold for the \emph{derangement polynomials} 
$$
d_n = \sum_{\sigma \in \mathfrak{D}_n }t^{ \mathrm{exc}(\sigma) }, \quad \quad \mathrm{exc}(\sigma)=|\{i : \sigma(i)>i\}|, 
$$
which count \emph{derangements} (fixed point free permutations) in the symmetric group $\mathfrak{S}_n$ by the number of \emph{excedances}, and the \emph{Eulerian polynomials} 
$$
A_n = \sum_{\sigma \in \mathfrak{S}_n }t^{ \mathrm{exc}(\sigma) }, 
$$
  see \cite[Corollary 1]{Brenti-90}. Hence we deduce that $d_x$ and $H_x$ are the traditional derangement polynomials and Eulerian polynomials, respectively. 
\end{example}

One would like an explicit interpretation of the function $d$ in terms of the $H$ polynomials. We will use the following non-recursive formula for $H$.
\begin{theorem}{\cite[Theorem 3.10]{ferroni-matherne-vecchi}}\label{thm:H from chains}
    Let $P$ be a weakly ranked poset and suppose $g \in I(P)$ satisfies \eqref{condition for kls}. For $x \leq y$ in $P$, 
    \[
    H_{x,y} = \sum_{x=z_0<z_1<\cdots <z_m \leq y} g_{z_m,y}\prod_{i=1}^{m} \left( \frac{\I_{\rho_{z_{i-1},z_i}}(g_{z_{i-1},z_i}) - t g_{z_{i-1},z_i}}{t-1}\right). 
    \]
   or, alternatively,
   \[
   H_{x,y} = \sum_{x=z_0<z_1<\cdots < z_m < y} \frac{\I_{\rho_{z_m,y}} (g_{z_m,y}) - g_{z_m,y}}{t-1}\prod_{i=1}^{m} \left( \frac{\I_{\rho_{z_{i-1},z_i}} (g_{z_{i-1},z_i}) - t g_{z_{i-1},z_i}}{t-1}\right).
   \]
\end{theorem}

When $g$ is scalar, this simplifies to the following result. 
\begin{corollary}\label{cor:H from chains}
    Suppose $g \in I(P)$ is scalar and $x\leq y$ in $P$. Then
    \begin{equation}\label{eq:def H up to top}
    H_{x,y} = \sum_{x=z_0<z_1<\cdots <z_m \leq y} g_{z_m,y} t^m\prod_{i=1}^{m} \left(g_{z_{i-1},z_i} \frac{t^{\rho_{z_{i-1},z_i}-1}-1}{t-1}\right)
    \end{equation}
    or, alternatively,
    \begin{equation}\label{eq: def H not up to top}
    H_{x,y} = \sum_{x=z_0<z_1<\cdots <z_m < y} g_{z_m,y} \frac{t^{\rho_{z_m,y}} -1}{t-1}t^m\prod_{i=1}^{m} \left(g_{z_{i-1},z_i}\frac{t^{\rho_{z_{i-1},z_i}-1}-1}{t-1}\right).
    \end{equation}
\end{corollary}
\begin{proof}
    We just observe that under these hypotheses $\I_{\rho_{x,y}}(g_{x,y}) = g_{x,y}t^{\rho_{x,y}}$.
\end{proof}

We provide an interpretation for $d$ in the case of scalar $g$. For a weakly ranked and bounded poset $P$ of rank $n$, we denote by $\tau(P)$ the \emph{truncation} of $P$, obtained from $P$ by removing all elements of rank $n-1$ and by setting the new rank of $\one$ to be equal to $n-1$. We keep denoting the rank function in the truncation by $\rho$ to not overload the notation.
If $g \in I(P)$ is scalar\footnote{The reason why we restrict to scalar functions $g$ is to avoid having degrees that are too large for intervals of $\tau(P)$. One could extend the same construction to a slightly larger set of functions on $I(P)$ but this goes outside the scope of this paper.}, then the function $g$ restricted to intervals of $\tau(P)$ is still a function satisfying \eqref{condition for kls}. In particular one can define Chow functions with respect to $g$ also in $\tau(P)$.
\begin{theorem}\label{thm: d as truncation}
    Suppose $g \in I(P)$ is scalar. Then
    \[
    d_{x,y} = \begin{cases}
        1, & \mbox{ if }  \rho_{x,y} = 0, \\
        0, & \mbox{ if }  \rho_{x,y} = 1, \mbox{ and } \\
        tH_{\tau[x,y]}, & \mbox{ if } \rho_{x,y} > 1.
    \end{cases}
    \]
    or, equivalently,
    \[
    H_{x,y} = g_{x,y} + t\sum_{\substack{x\leq z \leq y \\ \rho_{x,z} \geq 2}}H_{\tau[x,z]}g_{z,y}, \quad \mbox{ for all } x\leq y.
    \]
\end{theorem}
\begin{proof}
    That the two statements are equivalent follows by Theorem \ref{thm: def H and d}. We will prove the second. We only need to consider  the case when $\rho_{x,y} > 1$.
    By \eqref{eq:def H up to top}, if $m = 0$, the contribution of the sum is $g_{x,y}$. If $m > 0$, then there exists a maximal element $z_m$ of each chain, which we denote by $z$. By choosing an element $z$ and grouping all the terms in the sum that have $z_m = z$, the statement reduces to showing that $H_{\tau[x,z]}$ coincides with
    \[\sum_{x=z_0<z_1<\cdots <z_{m-1} < z} g_{z_{m-1},z} \frac{t^{\rho_{z_{m-1},z}-1}-1}{t-1}\ t^{m-1}\prod_{i=1}^{m-1} \left(g_{z_{i-1},z_i} \frac{t^{\rho_{z_{i-1},z_i}-1}-1}{t-1}\right).
    \]
    This follows from \eqref{eq: def H not up to top} by observing that the rank $\rho_{z_{m-1},z}$ in the truncation $\tau[x,z]$ drops by one.
\end{proof}

\begin{corollary}\label{cor: H alternative truncation formula}
    Let $g \in I(P)$ be scalar. Then
    \[
    H_{x,y} = (1+t)H_{\tau[x,y]} - t H_{\tau^2[x,y]} + t\sum_{\substack{x \leq z\leq y\\ \rho_{z,y}=1}} H_{\tau[x,z]}g_{z,y}.
    \]
\end{corollary}
\begin{proof}
    The statement follows by using Theorem \ref{thm: d as truncation} on $H_{x,y}$ and $H_{\tau[x,y]}$ and taking their difference. 
\end{proof}

\begin{remark}
    When $g = \zeta$, Theorem \ref{thm: d as truncation} specializes to \cite[Proposition~4.12]{ferroni-matherne-vecchi}, which reads
    \begin{equation}\label{HR-def-trunc}
        H_{x,y}= 1 + t\sum_{\substack{x\leq  z\leq y\\ \rho_{x,z} \geq 2}}H_{\tau([x,z])}.
    \end{equation}

\end{remark}
\subsection{Augmented Chow functions}
In \cite[Definition 3.13]{ferroni-matherne-vecchi} the \emph{augmented Chow function} was defined as the element in $I_\rho(P)$, 
\begin{equation}\label{Aug-Chow-def}
 \I(g) H.
\end{equation}
The next theorem offers an alternative equivalent definition.
\begin{theorem}\label{thm: def G and A}
    Let $g \in I(P)$ be a function satisfying \eqref{condition for kls}. Then there exist unique functions $G$ and $A$ in $I_\rho(P)$ for which 
    \begin{itemize}
        \item[(i)] $A_{x,x} = 1$ for each $x \in P$,
        \item[(ii)] $\I(G) = G$,
        \item[(iii)] $t\I(A) = A + (t-1)\delta$,
        \item[(iv)] $G = Ag$.
    \end{itemize}
    Moreover $G$ is the augmented Chow function \eqref{Aug-Chow-def} associated to $g$, and the polynomials $A_{x,y}$ satisfy the recursion
    \begin{equation}\label{A-recu}
    A_{x,y} = t\SSS_{\rho_{x,y}}\left(\sum_{x\leq z < y}A_{x,z}g_{z,y} \right), \quad x<y.
    \end{equation}
\end{theorem}
\begin{proof}
    Let $G = \I(g)H$ and $A = Gg^{-1} = \I(g) H g^{-1}= \I(g)d$. These satisfy (i) and (iv). To prove (ii) we first observe that 
    $$
    \kappa H = (\kappa-t\delta+t\delta) (1-t)(\kappa-t\delta)^{-1} = (1-t)\delta+tH=\I(H), 
    $$
    where again we work in the incidence algebra over $\RR(t)$, 
    and then write
    \[
    \I(G) = \I(\I(g)H) = g\I(H) = g\kappa H = \I(g)H = G.
    \]
    Also 
    \begin{align*}
    t\I(A) &= t gd = tgHg^{-1} = g(\I(H) - (1-t)\delta)g^{-1} = g \I(H) g^{-1} + (t-1)\delta
    \end{align*}
    and 
    \[
    g \I(H) g^{-1} = g\kappa H g^{-1} = g g^{-1}\I(g)Hg^{-1}= \I(g)Hg^{-1}= A, 
    \]
    which proves (iii). 
    
    For the uniqueness, we only need to prove that $A$ is unique. The uniqueness of $G$ then follows from $G = Ag$. To prove uniqueness of $A$ it remains to prove \eqref{A-recu}, the proof of which is almost identical to that of \eqref{d-recu-2}. 
\end{proof}

We call the unique function $A$ the \emph{Chow-Eulerian function} associated to $g$. When $P$ is bounded, $G_P = G_{\zero,\one}$ is called the \emph{augmented Chow polynomial} of $P$ (with respect to $g$) and $A_P = A_{\zero,\one}$ is called the \emph{Chow-Eulerian polynomial} of $P$.

The proof of the next corollary is the same as that of Corollary \ref{cor: H and d for initial intervals}. 
\begin{corollary}\label{aug-rec-cor}
    Let $P$ be a poset with a unique least element $\zero$ and $g \in I(P)$ be a function satisfying \eqref{condition for kls}. Then there are two unique families of polynomials $\{G_x\}_{x\in P}$ and $\{A_x\}_{x\in P}$ in $\R[t]$ that satisfy 
    \begin{itemize}
        \item[(i)] $A_{\zero} = 1$,
        \item[(ii)] $\I_{\rho_x}(G_x) = G_x$ for each $x \in P$,
        \item[(iii)] $\I_{\rho_x}(A_x) = A_x/t$ for each $x>\zero$, \mbox{ and }
        \item[(iv)] $G_y = \sum_{x \leq y}A_xg_{x,y}$  for each $y \in P$.
    \end{itemize}
    Moreover $G_x = G_{\zero,x}$ and $A_x= A_{\zero, x}$ for each $x \in P$, and 
      \begin{equation}\label{A-recu2}
    A_{y} = t\SSS_{\rho(y)}\left(\sum_{x< y}A_{x}g_{x,y} \right), \quad \zero<y.
    \end{equation}
\end{corollary}

\begin{example}
Suppose $P$ is again a boolean algebra of rank $r$, and $g= \zeta$. As in Example \ref{H-Eul-ex} it follows that there are polynomials $\{g_n\}_{n=0}^r$ and $\{f_n\}_{n=0}^r$ for which $G_x = g_{\rho(x)}$ and $A_x = f_{\rho(x)}$ for each $x \in P$. The identity $\I(A) = g\I(d) = gd$ then implies 
    \[
    \I(f_n) = \sum_{k=0}^n \binom{n}{k}d_{n-k} = \sum_{k=0}^n \binom{n}{k}d_{k},
    \]
    so $\I_n(f_n)=A_n$, the $n$th Eulerian polynomial by Example \ref{H-Eul-ex}. Hence $f_0=1$, and $f_n= tA_n$ for $n \geq 1$. The defining identity $G=Ag$ then implies 
    \[
    g_n= 1 + t\sum_{k=1}^n \binom n k A_{k}, \quad n \geq 0, 
    \]
    and by the uniqueness of $\{G_x\}_x$ we conclude that the polynomials $G_x$, $x \in P$, coincide with the \emph{binomial Eulerian polynomials} $\widetilde{A}_{\rho(x)}$ which are palindromic and satisfy the same defining identities (see the discussion in \cite[Section~10.4]{postnikov-reiner-williams}).
\end{example}

We shall now see that augmented Chow polynomials are themselves Chow polynomials. This extends \cite[Corollary~4.6]{ferroni-matherne-vecchi} where the case when $g =\zeta$ was proved. 
Let $P$ be a poset with a unique least element $\zero_P$. Denote by $(\aug(P),\rho_\aug)$ the \emph{augmentation} of $P$, i.e., the poset obtained from $P$ by adding a new least element $\zero_\aug$. The rank function $\rho_\aug$ is such that
\[
(\rho_\aug)_{x,y} = \begin{cases}
    \rho_{x,y}, &\mbox{ if } x\neq \zero_\aug, \mbox{ and }\\
    \rho_{\zero_P,y} + 1, &\mbox{ if } x = \zero_\aug \mbox{ and } \ y \neq \zero_\aug.
\end{cases}
\]
Define the element $\overline{g} \in I(\aug(P))$ by
\[
\overline{g}_{x,y} = \begin{cases}
    g_{x,y}, &\mbox{ if } x \neq \zero_\aug, \mbox{ and}\\
    g_{\zero_P,y}, &\mbox{ if } x= \zero_\aug \mbox{ and } \ y \neq \zero_\aug,\\
    1, &\mbox{ if }x= y = \zero_\aug.
\end{cases}
\]
\begin{theorem}\label{thm: G is H, A is d}
    Let $P$ be a poset with a unique least element $\zero_P$ and $g \in I(P)$ be a function satisfying \eqref{condition for kls} and let $\{G_x\}_{x \in P}$ and $\{A_x\}_{x \in P}$ be the augmented Chow polynomials and Chow-Eulerian polynomials associated to $P$, respectively. 
    
    If $\overline{g} \in I(\aug(P))$ is defined as above, then
    \begin{align*}
    G_x &= H_{\aug([\zero_P,x])}, \quad  \mbox{ for each $x \in P$,  and  } \\
     A_x &= d_{\aug([\zero_P,x])}, \quad  \mbox{ for each $x > \zero_P$ in $P$.   }
    \end{align*}
\end{theorem}
\begin{proof}
Let $\{H_x'\}_{x \in \aug(P)}$ and $\{d_x'\}_{x \in \aug(P)}$ be the Chow polynomials and Chow-derangement polynomials associated to $\overline{g}$ in $\aug(P)$, respectively. Then $\I_{\rho(x)}(H_x') = H_x'$ for each $x \in P$, and  $\I_{\rho(x)}(d_x') = d_x'/t$ for $x > \zero_P$ in $P$, since $\rho(x) = \rho_{\aug}(x)-1$. By Corollary \ref{cor: H and d for initial intervals}, for each $y \in P$, 
\begin{align*}
H_y' &= d_{\zero_{\aug}}'\overline{g}_{\zero_{\aug},y}+ \sum_{\zero_P \leq x \leq y} d_x'g_{x,y} \\
&= g_{\zero_P,y}+ \sum_{\zero_P < x \leq y} d_x'g_{x,y} \quad \mbox{ (since $d_{\zero_P}'=0$) } \\
&= \sum_{\zero_P \leq x \leq y} d_x''g_{x,y},
\end{align*}
where $d_x'':= d_x'$ unless $x= \zero_P$, and $d_{\zero_P}'' :=1$. From Corollary \ref{aug-rec-cor} it now follows that $H'_x = G_x$ and $d_x''=A_x$ for each $x \in P$, which proves theorem. 
\end{proof}

\begin{corollary}\label{cor: G formula with truncation}
    Let $g\in I(P)$ be scalar. Then
    \[
    G_y = g_{\zero,y} + t\sum_{\zero < x \leq y}G_{\tau[\zero,x]}g_{x,y}, \quad \mbox{ for each } y \in P.
    \]
\end{corollary}
\begin{proof}
    This is a direct consequence of Theorem \ref{thm: G is H, A is d} and Theorem \ref{thm: d as truncation}.
\end{proof}

\subsection{Duality for Chow functions}
If $P$ is a finite weakly ranked poset, then so is its dual $P^*$, i.e., the poset on the same ground set as $P$ for which $x \leq_{P^*} y$ if and only if $y \leq_{P} x$. If $f  \in I(P)$, let $f^T \in I(P^*)$ be defined by $f^T_{y,x} = f_{x,y}$ for all $x \leq y$ in $P$. If $g$ is a left KLS function for $P$, then 
$g^T$ is a left KLS function for $P^*$. Denote by  $H^*, d^*, G^*$ and $A^*$, the corresponding Chow functions for $P^*$
\begin{theorem}\label{dualfunc}
Suppose $P$ is a finite weakly ranked poset with a left KLS function $g$. Then 
$$
A^*=\I(H^T), \ \ H^*=\I(A^T) \ \ \mbox{ and } \ \ G^*= G^T. 
$$
\end{theorem}
\begin{proof}
By Theorem \ref{thm: def G and A} for $P$, 
$$
G= \I(g)H=Ag, \ \  \I(G)=G \ \ \mbox{ and } \ \ t\I(A)= A+(t-1)\delta. 
$$
Transpose these identities and apply $\I$ to deduce 
$$
G^T= \I(H^T)g^T = \I(g^T)\I(A^T), \ \ \I(G^T)=G^T \ \  \mbox{ and } \ \ t\I(\I(H)^T)=\I(H)^T+(t-1)\delta, 
$$
where we have used Theorem \ref{thm: def H and d} for the last equation. The theorem now follows from the uniqueness in Theorem \ref{thm: def G and A} for $P^*$. 
\end{proof}

\section{Chow polynomial of matrices and weak-rank uniform posets}\label{section chow matrix}
Let $R=(r_{n,k})_{n,k=0}^N$, where $N \in \NN \cup \{\infty\}$, be a lower triangular matrix with all diagonal entries equal to one. Consider the chain $[0,N]=\{0<1<\cdots<N\}$ graded by the rank function $\rho(n) = n$. Then the matrix $R$ corresponds to an element $g = g[R]$ in the incidence algebra of $[0,N]$ defined by
\[
g_{k,n} = r_{n,k}, \quad \mbox{ for all } 0\leq k \leq n \leq N.
\]
Clearly $g$ satisfies \eqref{condition for kls} and therefore we can define the corresponding Chow, Chow-derangement, augmented Chow, and Chow-Eulerian polynomials. Corollaries \ref{cor: H and d for initial intervals} and \ref{aug-rec-cor} then translate as the following corollaries which provide the defining relations for these polynomials. 

\begin{corollary}\label{cor: H and d for matrices}
 Let $R=(r_{n,k})_{n,k=0}^N$, $N \in \NN \cup \{\infty\}$, be a lower triangular matrix with all diagonal entries equal to one. There are unique polynomials $\{d_n\}_{n=0}^N$ and $\{H_n\}_{n=0}^N$ for which 
\begin{itemize}
    \item[(i)] $d_0=1$, 
    \item[(ii)]$\I_n(H_n)=tH_n$, for each $1 \leq n \leq N$, 
    \item[(iii)] $\I_n(d_n) =d_n$, for each $0 \leq n \leq N$, and 
    \item[(iv)] $H_n=\sum_{k=0}^n r_{n,k} d_k$, for each $0 \leq n \leq N$. 
\end{itemize} 
Moreover $d_n=d_n[R]$ and $H_n=H_n[R]$, $n \leq N$,  are the Chow-derangement polynomials and the Chow polynomials associated to $g[R]$, respectively, and 
\begin{equation}\label{dn-recu}
d_n= t\SSS_{n-1}\left(\sum_{k=0}^{n-1} r_{n,k} d_k \right), \quad 1 \leq n \leq N. 
\end{equation}
\end{corollary}

\begin{corollary}\label{cor: G and A for matrices}
 Let $R=(r_{n,k})_{n,k=0}^N$, $N \in \NN \cup \{\infty\}$, be a lower triangular matrix with all diagonal entries equal to one. There are unique polynomials $\{A_n\}_{n=0}^N$ and $\{G_n\}_{n=0}^N$ for which 
\begin{itemize}
    \item[(i)] $A_0=1$, 
    \item[(ii)]$\I_n(G_n)=G_n$, for each $0 \leq n \leq N$, 
    \item[(iii)] $\I_n(A_n) =A_n/t$, for each $1 \leq n \leq N$, and 
    \item[(iv)] $G_n=\sum_{k=0}^n r_{n,k} A_k$, for each $0 \leq n \leq N$. 
\end{itemize} 
Moreover $A_n=A_n[R]$ and $G_n=G_n[R]$, $n \leq N$, are the Chow-Eulerian polynomials and the augmented Chow polynomials associated to $g[R]$, respectively, and 
$$
A_n= t\SSS_{n}\left(\sum_{k=0}^{n-1} r_{n,k} A_k \right), \quad 1 \leq n \leq N. 
$$
\end{corollary}

In Section \ref{section toeplitz} we will need the following result which tells us how the Chow polynomials behave under conjugation by a diagonal matrix. 

\begin{proposition}\label{scaling}
Let $R=(r_{n,k})_{n,k=0}^N$, $N \in \NN \cup \{\infty\}$, be a lower triangular matrix with entries in a field $K$, and all diagonal entries equal to one. Let further $\{c_n\}_{n=0}^N$ be a sequence of nonzero elements in $K$ with $c_0=1$. Consider the matrix $R'= (r_{n,k} c_nc_k^{-1})_{n,k=0}^N$. 

Then 
\begin{alignat*}{2}
d_n[R']&= c_n d_n[R], \quad &H_n[R']= c_n H_n[R], \\
A_n[R']&= c_n A_n[R],\; \mbox{ and }\; &G_n[R'] = c_n G_n[R]. 
\end{alignat*}
\end{proposition}
\begin{proof}
By Corollary \ref{cor: H and d for matrices},
$$
H_n[R']=\sum_{k=0}^n r_{n,k}c_n c_k^{-1} d_k[R'], \quad n \leq N. 
$$  
Hence 
$$
H_n[R']c_n^{-1}=\sum_{k=0}^n r_{n,k} d_k[R']c_k^{-1}, \quad n \leq N. 
$$  
The uniqueness in Corollary \ref{cor: H and d for matrices} now implies that, for each $n$, $d_n[R]= d_n[R']c_n^{-1}$ and $H_n[R]= H_n[R']c_n^{-1}$. The proof for $G_n$ and $A_n$ follows similarly.  
\end{proof}

Next we will see how  Corollary \ref{cor:H from chains} gives a non-recursive formula to compute $H_n$.
For $S = \{ s_1 < s_2 < \cdots < s_m\} \subseteq [n-1]$ let
\[
\alpha_{[0,n]}(S) = \prod_{i=1}^{m+1} r_{s_i,s_{i-1}},
\]
where $s_0 = 0$ and $s_{m+1} = n$.
\begin{proposition}\label{def:HR canonical decomposition}
    Let $R=(r_{n,k})_{n,k=0}^N$ be a lower triangular matrix with all diagonal entries equal to one. Then $H_0 = 1$ and, for $n \geq 1$,
     \[
   H_n=    \sum_{m\geq 0} \! \! \sum_{\substack{S\subseteq [n] \\ S = \{s_1 < \cdots < s_m\} }} \! \! \! \! \! \! \!  t^m\alpha_{[0,n]}(S)\prod_{i=1}^m\left(\frac{t^{s_i - s_{i-1} - 1} - 1}{t-1} \right)
    \]
    or, equivalently,
    \[
   H_n=    \sum_{m\geq 0} \! \! \sum_{\substack{S\subseteq [n-1] \\ S = \{s_1 < \cdots < s_m\} }} \! \! \! \! \! \! t^m\alpha_{[0,n]}(S)\frac{t^{n-s_m}-1}{t-1}\prod_{i=1}^m\left(\frac{t^{s_i - s_{i-1} - 1} - 1}{t-1} \right),
    \]
where  $s_0 = 0$.
\end{proposition}

Let $\partial_n R$ denote the matrix obtained from $R$ by deleting the row and column indexed by $n$. Theorem \ref{thm: d as truncation} implies that the Chow-derangement polynomials asssociated to $R$ may be expressed as 
\begin{equation}\label{dntruncR}
    d_n = \begin{cases}
        1, &\mbox{if } n=0,\\
        0, &\mbox{if } n=1,\\
        tH_{n-1}[\partial_{n-1}R], &\mbox{if } n>1.
    \end{cases}
\end{equation}
    In particular,
    \[
    H_n = r_{n,0} + t \sum_{k=2}^n r_{n,k}H_{k-1}[\partial_{k-1}R], \quad n\geq 0.
    \]

Let also $\overline{R}=(\overline{r}_{n,k})_{n,k=0}^{N+1}$ be the matrix obtained form $R$ by adding a new initial row 
$[1,0,0,\ldots]$ and a new initial column $[1,r_{0,0}, r_{1,0}, r_{2,0}, \ldots]$. Then, by Theorem \ref{thm: G is H, A is d}, the augmented Chow polynomials $G_n$ of $R$ satisfy
\begin{equation}\label{aug-matrx}
G_n= H_{n+1}[\overline{R}], \quad n \geq 0.
\end{equation}

Similarly, the Chow-Eulerian polynomials $A_n$ satisfy
$A_0=1$, and 
\begin{equation}\label{DR-def}
A_n = t  H_n[\partial_n \overline{R}], \quad n \geq 1.
\end{equation}

We shall now see that the Chow polynomials of certain matrices are actually characteristic Chow polynomials of posets, by following the construction in \cite[Section 5]{branden-saud-1}. We say that a weakly ranked poset $P$ with a least element $\zero$ is \emph{weak-rank uniform} if for any $x$ and $y$ in $P$ with $\rho(x) = \rho(y)$, 
\begin{equation}\label{wru}
|\{z \leq x : \rho(z) = k\}| = |\{z \leq y : \rho(z) = k\}|, \quad \text{ for each }0\leq k \leq \rho(x).
\end{equation}
If $\rho(x) = n$, we define $r_{n,k}=r_{n,k}(P) = |\{z \leq x : \rho(z) = k \}|$, and write $R = R(P) = (r_{n,k})_{n,k\geq 0}$. Notice that $r_{n,0} = r_{n,n} = 1$ for each $n$. If $P$ is graded and satisfies \eqref{wru}, then we say that $P$ is \emph{rank-uniform}. 
\begin{proposition}\label{matrix-poset}
    Let $P$ be a weak-rank uniform poset, and let $R = (r_{n,k}(P))_{n,k=0}^N$ be the corresponding matrix, where $N\in \NN \cup \{\infty\}$. If $n \leq N$, then $H_n[R]$ is equal to the characteristic Chow polynomial $H_{x}$, where $x$ is any element of $P$ of rank $n$.
\end{proposition}
\begin{proof}
Let $\{d'_x\}_{x \in P}$ and  $\{H'_x\}_{x\in P}$ be defined by 
$$
d'_x= d_{\rho(x)}[R] \quad  \mbox{ and } \quad H'_x= H_{\rho(x)}[R], \quad x\in P. 
$$
Then $\{d'_x\}_{x \in P}$ and  $\{H'_x\}_{x\in P}$ satisfy the conditions in Corollary \ref{cor: H and d for initial intervals} by construction and Corollary \ref{cor: H and d for matrices}. Hence the proof follows from the uniqueness in Corollary \ref{cor: H and d for initial intervals}. 
\end{proof}

\section{Totally nonnegative matrices and their Chow-polynomials}\label{TNChowsection}

The main result of this section is Theorem \ref{mainman2}, which says that the Chow polynomial of any lower triangular and totally nonnegative matrix with all diagonal entries equal to one is real-rooted. We shall first prepare for the proof of this result by introducing relevant notation and proving some preliminary results on interlacing sequences of polynomials. 

\subsection{Totally nonnegative and resolvable matrices}
We start by collecting some notation and results from \cite{branden-saud-1}. Recall that a matrix with real entries is \emph{totally nonnegative} (or $\TN$) if all of its minors are nonnegative. 
\begin{definition}
\label{resolv}
Let $R=(r_{n,k})_{n,k=0}^N$, where $N \in \NN \cup\{\infty\}$,  be a lower triangular matrix with real entries whose diagonal entries are all equal to one, and let $R_n = \sum_{k=0}^n r_{n,k} t^k$ be the generating polynomial of the $n$th row. 
The matrix $R$ is called \emph{resolvable}\footnote{Resolvable matrices are always lower triangular matrices with \emph{real} entries, whose diagonal entries are all equal to one.} if there is an array  $(\lambda_{n,k})_{0\leq k\leq n <N}$ of nonnegative numbers, and an array of monic polynomials $(R_{n,k})_{0\leq k \leq n \leq N}$ in $\RR[t]$ for which 
\begin{itemize}
\item $R_{n,0}=R_n$ and $R_{n,n} =t^n$ for each $0\leq n \leq N$,
\item $t^k$ divides  $R_{n,k}$ for all $0 \leq k \leq n \leq N$, and 
\item if $0\leq k \leq n <N$, then 
\begin{equation}\label{r-pasc}
R_{n+1,k}=R_{n+1,k+1}+ \lambda_{n,k} R_{n,k}. 
\end{equation}
\end{itemize}
\end{definition}
If the matrix $R$ is resolvable, then we say that the polynomials $(R_{n,k})_{0\leq k \leq n \leq N}$ \emph{resolve} $R$.
\begin{example}
The identity matrix is resolvable, with $R_{n,k}=t^n$ and $\lambda_{n,k}=0$ for all $0 \leq k \leq n$. 

Pascal's triangle $(\binom n k)_{n,k=0}^\infty$ is resolvable, with $R_{n,k}=t^k(1+t)^{n-k}$ and $\lambda_{n,k}=1$ for all $0 \leq k \leq n$. 

For more examples we refer to \cite{branden-saud-1,branden-saud-2}. 
\end{example}
Notice that if $R$ is resolvable, then by \eqref{r-pasc},
\begin{equation}
\label{r-sum}
    R_{n+1,k}= t^{n+1}+ \sum_{j \geq k} \lambda_{n,j}R_{n,j}, \ \ \ \ 0\leq k \leq n<N.
\end{equation}
The next theorem, proved in \cite[Theorem 2.6]{branden-saud-1}, characterizes resolvability in terms of totally nonnegative matrices and ``quantum'' real-rooted polynomials. 
\begin{theorem}\label{eqcon}
    Let $R=(r_{n,k})_{n,k=0}^N$ be a lower triangular matrix with all diagonal entries equal to one. The following are equivalent:
    \begin{itemize}
        \item[(i)] $R$ is resolvable,
        \item[(ii)] There are linear diagonal operators $\alpha_i : \RR[t] \to \RR[t]$, $1 \leq i \leq N$,  such that 
        $$
            \alpha_i(t^k)= \alpha_{i,k}t^k, \mbox{ where } \alpha_{i,k} \geq 0 \mbox{ for all } i,k, 
        $$
        and 
        $$
            R_n = (t+ \alpha_1)(t+\alpha_2) \cdots (t+\alpha_n) 1.
        $$
        \item[(iii)] $R$ is $\mathrm{TN}$. 
    \end{itemize} 
    Moreover if (ii) is satisfied, then $R_{n,k} = (t+\alpha_1)\cdots (t+\alpha_{n-k})t^k$ and $\lambda_{n,k}= \alpha_{n+1-k,k}$. 
\end{theorem}

\subsection{\texorpdfstring{Interlacing zeros and $\mathcal I_n$-interlacing sequences of polynomials}{Interlacing zeros and In-interlacing sequences of polynomials}}
Recall that a polynomial $p \in \RR[t]$ is called \emph{real-rooted} if all of its zeros are real. For technical reasons, the zero polynomial is considered to be real-rooted. Suppose $p$ and $q$ are two real-rooted polynomials in $\RR[t]$ with zeros $\cdots \leq \beta_2 \leq \beta_1$ and $\cdots \leq \alpha_2 \leq \alpha_1$, respectively.
The zeros \emph{interlace} if 
$$
  \cdots \leq \beta_2 \leq \alpha_2 \leq \beta_1\leq \alpha_1 \ \ \ \mbox{ or }  \ \ \      \cdots \leq \alpha_2 \leq \beta_2 \leq \alpha_1 \leq \beta_1. 
$$
If $p$ and $q$ have interlacing zeros, then the \emph{Wronskian} $W[p,q]:=p'q-pq'$ is either nonpositive on $\RR$ or nonnegative on $\RR$. 
We write $p \prec q$ if $p$ and $q$ are real-rooted, their zeros interlace, and $W[p,q] \leq 0$ on all of $\RR$. 
For technical reasons we consider the identically zero polynomial to be real-rooted and write $0 \prec p$ and $p \prec 0$ for any other real-rooted polynomial $p$. 
\begin{remark}\label{altint}
If the signs of the leading coefficients of two real-rooted polynomials $p$ and $q$  are positive, then $p \prec q$ if and only if 
$$
  \cdots \leq \beta_2 \leq \alpha_2 \leq \beta_1\leq \alpha_1,
$$
where $\cdots \leq \beta_2 \leq \beta_1$ and $\cdots \leq \alpha_2 \leq \alpha_1$ are the zeros of $p$ and $q$, respectively. Indeed, by continuity, we may assume that $p$ and $q$ have no common zeros. 
The correct sign of the Wronskian is read off when it is evaluated at the largest zero of $pq$. 
\end{remark}

A polynomial $f \in \CC[t]$ is called \emph{stable} if either $f \equiv 0$ or all the zeros of $f$ have nonpositive imaginary parts. 
The Hermite-Biehler theorem relates stability to interlacing zeros, for a proof see \cite{wagnersurvey}\footnote{There is a typo in the definition of proper position in \cite{wagnersurvey}.  It is essential to add the condition that the zeros of $f$ and $g$ interlace.}.  
\begin{theorem}[Hermite-Biehler Theorem]
\label{thm: Hermite-Biehler}
Let $p, q \in \RR[t]$. Then $p \prec q$ if and only if $q+ip$ is stable.  
\end{theorem}
One consequence of Theorem~\ref{thm: Hermite-Biehler} is that $p$ and $q$ in $\RR[t]$ are real-rooted whenever $q+ip$ is stable.  
Moreover, it follows that for real-rooted polynomials $p$ and $q$ in $\RR[t]$:
\begin{itemize}
\item $p \prec \alpha p$ for all $\alpha \in \RR$,
\item $p\prec q$ if and only if $-q \prec p$,
\item $p\prec q$ if and only if $\alpha p \prec \alpha q$, for some $\alpha \in \RR \setminus \{0\}$. 
\end{itemize}
The following lemma follows directly from Remark \ref{altint}. 
\begin{lemma}\label{addz}
Suppose $p,q$ are real-rooted polynomials with positive leading coefficients, and let $\alpha \in \RR$ be greater or equal to the largest zero of $q$. Then $p \prec q$ if and only if $q \prec (t-\alpha)p$. 
\end{lemma}

We will make frequent use of the following proposition.
\begin{proposition}[See e.g. Lemma 2.6 in \cite{BB}]\label{cones}
Let $p$ be a real-rooted polynomial that is not identically zero. The sets 
$$
\{q \in \RR[t] : q \prec p\} \ \ \mbox{ and }  \ \ \{q \in \RR[t] : p \prec q\}
$$
are convex cones. 
\end{proposition}
A sequence of polynomials $f_1, \ldots, f_m \in \RR_{\geq 0}[t]$ is called \emph{interlacing} if $f_i \prec f_j$ for all $i<j$. 
The following elementary lemma is useful when proving that sequences of polynomials are interlacing, see e.g. \cite{wagner}. 
\begin{lemma}\label{wagner}
Suppose $f_1, f_2, \ldots, f_n$ are real-rooted polynomials with positive leading coefficients. If $f_1 \prec f_2 \prec \cdots \prec f_n$ and $f_1 \prec f_n$, then $f_i \prec f_j$ for all $1\leq i<j\leq n$. 
\end{lemma}

We will now define a central notion of this paper. 

\begin{definition}\label{def: I_n-int}
    We say that a sequence $f_1, f_2, \ldots, f_m $ of polynomials in $\RR_{\geq 0}[t]$  of degree at most $n$ is  $\I_n$-\emph{interlacing} if 
$$
f_1, f_2, \ldots, f_m, \I_n(f_m), \I_n(f_{m-1}), \ldots, \I_n(f_1)
$$
is an interlacing sequence of polynomials. 
\end{definition}

The next lemma provides some simple operations that preserve  interlacing and $\I_n$-interlacing sequences. 
\begin{lemma}\label{merge}
Suppose $f_1, f_2, \ldots, f_m$ is an interlacing ($\I_n$-interlacing) sequence of polynomials. Then the sequences 
\begin{itemize}
\item $a_1f_1, a_2f_2, \ldots, a_mf_m$, where $a_i \geq 0$ for each $i$, 
\item $f_1, \ldots, f_{i-1}, f_{i+1}, f_{i+2}, \ldots, f_m$, for each $i$, and 
\item $f_1, \ldots,f_i, f_i + f_{i+1}, f_{i+1}, f_{i+2}, \ldots, f_m$ 
\end{itemize}
are interlacing ($\I_n$-interlacing). 
\end{lemma}
\begin{proof}
The first two claims follow directly from the definitions. The third one follows from Proposition \ref{cones}.
\end{proof}

\newcommand{\II}{\mathbb{I}}
Let $\mathbb{I}_n$ be the set of all polynomials $f \in \RR_{\geq 0}[t]$ of degree at most $n$ for which $f \prec I_n(f)$. 

\begin{lemma}\label{snf}
If $f,g$ is an $\I_n$-interlacing sequence, then $\SSS_n(f) \in \RR_{\geq 0}[t]$ and $\SSS_n(f) \prec g$. 
\end{lemma}

\begin{proof}
Suppose $f,g$ is an $\I_n$-interlacing sequence. We claim that either $\SSS_n(f) \equiv 0$, or the leading coefficient of $\SSS_n(f)$ is positive. Since $f\prec \I_n(f)$, the degree of $f$ is smaller or equal to the degree of $\I_n(f)$. If the degree of $\I_n(f)$ is strictly smaller than the degree of $f$, then the leading coefficient of $\SSS_n(f)$ equals that of $\I_n(f)$ and is thus positive. 

Otherwise if the degrees of $f$ and $\I_n(f)$ are equal, then write $f=t^M h$ where $h(0) \neq 0$, and suppose $h$ has degree $m=n-M-N$. Then 
$$
\I_n(f)= t^N \I_{n-M-N}(h)= t^N\I_{m}(h).
$$
Since the degrees of $f$ and $\I_n(f)$ agree, it follows that $M=N$. Also $h \in \II_m$, and since $\SSS_n(f)=t^N\SSS_m(h)$, the leading coefficient of $\SSS_n(f)$ is equal to the leading coefficient of $\SSS_m(h)$. Write 
$$
h=C(t+\alpha_1)(t+\alpha_2) \cdots (t+\alpha_m) 
$$
where $C >0$, and 
$$
\alpha_m \geq 1/\alpha_1 \geq \alpha_{m-1} \geq 1/\alpha_2 \geq \cdots \geq \alpha_1 \geq 1/\alpha_m >0,  
$$
since $h \in \II_m$. From this we deduce 
$$
\alpha_1 \cdots \alpha_m \geq 1/(\alpha_1 \cdots \alpha_m), 
$$
i.e., $\alpha_1 \cdots \alpha_n \geq 1$, with equality if and only if $\SSS_n(f)\equiv 0$. Notice that 
the leading coefficient of $\SSS_m(h)$ is equal to $C\alpha_1 \cdots \alpha_n -C$, which proves the claim. 

To prove $\SSS_n(f) \prec g$, we prove 
$$
\I_n(f)- f \prec (t-1)g. 
$$
Since $g \prec -f$ and $g \prec \I_n(f)$, Proposition \ref{cones} implies $g \prec \I_n(f)-f$. Notice that all zeros but $1$ of the polynomial $\I_n(f)-f$ are nonpositive, since  $g\prec \I_n(f)-f$ and $\SSS_n(f)$ has positive leading coefficient. It follows from Lemma \ref{addz} that $\I_n(f)- f \prec (t-1)g$. Also $\SSS_n(f)$ has nonnegative coefficients since its leading coefficient is nonnegative and $\SSS_n(f) \prec g$, where all the zeros of $g$ are nonpositive. 
\end{proof}

\begin{lemma}\label{daint}
Suppose $f,g$ is $\I_n$-interlacing. Then the sequence 
$$
\SSS_n(f), f, g,  t\SSS_n(g)+g 
$$
is $\I_n$-interlacing, that is, the sequence 
$$
\SSS_n(f), f, g,  t\SSS_n(g)+g, \I_n(g), \I_n(f), t\SSS_n(f)
$$ 
is interlacing.
\end{lemma}

\begin{proof}
By Lemma \ref{snf}, Lemma \ref{addz} and Proposition \ref{cones}, 
$$
\SSS_n(f) \prec f \prec g \prec t\SSS_n(g)+g = \I_n(t\SSS_n(g)+g) \prec \I_n(g) \prec \I_n(f) \prec \I_n(\SSS_n(f))= t \SSS_n(f), 
$$
from which the lemma follows using Lemma \ref{wagner}. 
\end{proof}

\begin{theorem}\label{maineng}
Suppose $\{f_k\}_{k=0}^m$ is $\I_n$-interlacing. Then the sequence $\{g_k\}_{k=0}^{m+1}$ defined by 
$$
g_k = t \SSS_n\left(\sum_{j=0}^m f_j\right) + \sum_{j \geq k}f_j
$$
is $\I_{n+1}$-interlacing.
\end{theorem}

\begin{proof}
We first prove that $g_0, \ldots, g_{m+1}$ is interlacing. Let $0\leq k < \ell \leq m+1$, and let 
$$
h_0= \sum_{j<k}f_j, \ \ h_1= \sum_{j=k}^{\ell-1} f_j \ \ \mbox{ and } \ \ h_2= \sum_{j\geq \ell} f_j.
$$
Then 
$$
g_k= t \SSS_n(h_0+h_1+h_2) +h_1+h_2 \ \ \mbox{ and } \ \ g_\ell= t \SSS_n(h_0+h_1+h_2) +h_2. 
$$
By Lemma \ref{merge} the sequence $h_0, h_1, h_2$ is $\I_n$-interlacing. 
We should prove 
$$
 t \SSS_n(h_0+h_1+h_2) +h_1+h_2  \prec t \SSS_n(h_0+h_1+h_2) +h_2. 
$$
By Proposition \ref{cones},
this reduces to proving 
$$
h_1 \prec  t \SSS_n(h_0+h_1+h_2) +h_2, 
$$
 which, by Proposition \ref{cones} again, reduces to proving 
$$
h_1 \prec  t \SSS_n(h_0),  h_1 \prec  t \SSS_n(h_1) \mbox{ and } h_1 \prec  t \SSS_n(h_2) +h_2, 
$$
which follows from Lemmas \ref{addz} and  \ref{daint}. Hence  $g_0, \ldots, g_{m+1}$ is interlacing. 

Notice that for  polynomials of degree at most $d$, $f \prec g$ if and only if $\I_d(g) \prec \I_d(f)$. Since $\I_{n+1}(g_{m+1})= g_{m+1}$, 
$$
g_0\prec g_1\prec  \cdots \prec g_{m+1} \prec \I_{n+1}(g_{m+1}) \prec \cdots \prec \I_{n+1}(g_0). 
$$
By Lemma \ref{wagner}, it remains to prove $g_0 \prec \I_{n+1}(g_0)$. However $g_0= t\SSS_n(f) + f$, where 
$f= \sum_{j=0}^m f_j$. Hence $\I_{n+1}(g_0) = t \I_{n}(g_0)= tg_0$, which concludes the proof. 
\end{proof}

\subsection{Chow-deranged maps and Chow-Eulerian transformations}\label{maps}
Let  $R=(r_{n,k} )_{n,k=0}^N$, where $N \in \NN\cup\{\infty\}$, be a lower triangular matrix, with entries in $\R$, and  with all diagonal entries equal to one. Let further $\R_N[t]$ be the $\R$-module of all polynomials in $\R[t]$ of degree at most $N$. We define two $\R$-linear maps $\DD =\DD_R : \R_N[t] \to \R_N[t]$ and 
$\A =\A_R: \R_N[t] \to \R_N[t]$
$$
\DD(t^n)= d_n  \quad \mbox{ and } \quad \A(t^n) = A_n, \quad n \leq N, 
$$
where $d_n$ and $A_n$ are the Chow-derangement polynomials and Chow-Eulerian polynomials associated to $R$, respectively. These maps are called the \emph{Chow-deranged map} and the \emph{Chow-Eulerian map}, respectively. Notice that (iv) in Corollaries \ref{cor: H and d for matrices} and \ref{cor: G and A for matrices}   translate to
\[
H_n = \DD(R_n) \quad  \mbox{ and } \quad G_n= \A(R_n).
\]
\begin{remark}
    For $R_0=(\binom n k )_{n,k=0}^\infty$, the deranged map was first studied by the first author and Solus \cite{branden-solus}, who proved Theorem~\ref{mainman} below for the case when $R=R_0$. The Chow-Eulerian map for the case when $R=R_0$ was first considered by Brenti \cite{Brenti89}, who conjectured that $\A_{R_0}(f)$ is real-rooted whenever all zeros of $f \in \RR[t]$ are real and nonpositive. This conjectured was disproved in \cite{branden-jochemko}, where it was conjectured that $\A_{R_0}(f)$ is real-rooted whenever $f$ has a nonnegative expansion in $\{t^k(1+t)^{n-k}\}_{k=0}^n$. This conjecture was proved by Athanasiadis in \cite{athanasiadis-eulerian2}. 
\end{remark}

Theorem~\ref{mainmanA} below generalizes Athanasiadis result to any resolvable matrix.
Suppose $R$ is resolvable, and define 
$$
d_{n,k} = \DD(R_{n,k}) \quad \mbox{ and } A_{n,k} = \A(R_{n,k}), \quad 0\leq k \leq n \leq N. 
$$
\begin{lemma}\label{dnkrec}
Let $R$ be a resolvable matrix. For $0\leq k \leq n+1 \leq N$, 
\begin{align*}
    d_{n+1,k} &= t\SSS_n\left(\sum_{j \geq 0}\lambda_{n,j}d_{n,j}\right) + \sum_{j \geq k}\lambda_{n,j}d_{n,j}, \quad \mbox{ and }\\
    A_{n+1,k} &= t\SSS_{n+1}\left(\sum_{j \geq 0}\lambda_{n,j}A_{n,j}\right) + \sum_{j \geq k}\lambda_{n,j}A_{n,j}. 
\end{align*}
\end{lemma}

\begin{proof}
  By \eqref{r-sum} and Corollary \ref{cor: H and d for matrices}, 
\begin{align*}
d_{n+1,k} &= \DD(t^{n+1} + R_{n+1,k} -t^{n+1}) \\
                &= t\SSS_n\DD(R_{n+1,0} -t^{n+1})+ \DD(R_{n+1,k} -t^{n+1}) \\
                &= t\SSS_n\left(\sum_{j \geq 0}\lambda_{n,j}d_{n,j}\right) + \sum_{j \geq k}\lambda_{n,j}d_{n,j}. 
\end{align*}
The proof of the second recursion is almost identical. 
\end{proof}

\begin{theorem}\label{mainman} Let $R$ be a resolvable matrix, and $0\leq n \leq N$. Then 
$$
d_{n,0}, d_{n,1}, \ldots, d_{n,n}
$$
is an $\I_n$-interlacing sequence. 

Moreover if $f= \sum_{k=0}^n h_k R_{n,k}$, where $h_k \geq 0$ for each $k$, then $\DD(f)$ is real-rooted and $d_{n,0} \prec \DD(f) \prec d_{n,n}$. 
\end{theorem}

\begin{proof}
The proof of the first statement is by induction over $n$, the case when $n=0$ being clear. 
By induction the sequence $\{\lambda_{n,j}d_{n,j}\}_{j=0}^n$ is $\I_n$-interlacing. The proof of the first statement now follows by 
Lemma \ref{dnkrec}, Theorem \ref{maineng} and induction. 

The proof of the second statement from the first in conjunction with Proposition~\ref{cones}. 
\end{proof}
The proof of the next theorem is identical to that of the previous. 

\begin{theorem}\label{mainmanA} Let $R$ be a resolvable matrix, and $0\leq n \leq N$. Then 
$$
A_{n,0}, A_{n,1}, \ldots, A_{n,n}
$$
is an $\I_{n+1}$-interlacing sequence. 

Moreover if $f= \sum_{k=0}^n h_k R_{n,k}$, where $h_k \geq 0$ for each $k$, then $\A(f)$ is real-rooted and $A_{n,0} \prec \A(f) \prec A_{n,n}$. 
\end{theorem}

The first important consequence of Theorem \ref{mainman} establishes the real-rootedness of the Chow polynomials and Chow-derangement polynomials of any lower triangular $\TN$-matrix with all diagonal entries equal to one. 
\begin{theorem}\label{mainman2}
Let $R = (r_{n,k})_{n,k=0}^N$ be a lower triangular $\TN$-matrix with all diagonal entries equal to one. Then the Chow polynomials $\{H_n\}_{n=0}^N$ and the Chow-derangement polynomials $\{d_n\}_{n=0}^N$ are real rooted. 

Moreover $H_n \prec d_n$, $H_n \prec H_{n+1}$ and  $d_n \prec d_{n+1}$ for each $n$. 
\end{theorem}
\begin{proof}
Since $d_{n,0}=H_n$ and $d_{n,n}=d_n$, the proof of all statement except $d_n \prec d_{n+1}$ and $H_n\prec H_{n+1}$ follow from Theorem~\ref{mainman}. Let $f= \sum_{j \geq 0} \lambda_{n,j} d_{n,j}$. Then, by Theorem~\ref{mainman}, Lemma \ref{cones} and Lemma \ref{daint}
$$
 H_n=d_{n,0} \prec f \prec t\SSS_n(f)+f = H_{n+1} = \I_n(H_{n+1}) \prec \I_n(f) \prec I_n(d_{n,0})= tH_n. 
$$ 
Since $H_n \prec t H_n$, Lemma \ref{wagner} implies $H_n \prec H_{n+1}$. 

Let  $g = d_n=d_{n,n}$. Then $f,g$ is an $\I_n$-interlacing sequence by Theorem \ref{mainman} and Lemma \ref{merge}. By Lemma \ref{daint}, 
$$
d_n=g \prec t\SSS_n(f)= d_{n+1}.
$$

\end{proof}
Again, the proof of the augmented version of Theorem \ref{mainman2} is identical. 
\begin{theorem}\label{mainman2A}
Let $R = (r_{n,k})_{n,k=0}^N$ be a lower triangular $\TN$-matrix with all diagonal entries equal to one. Then the augmented Chow polynomials $\{G_n\}_{n=0}^N$ and the Chow-Eulerian polynomials $\{A_n\}_{n=0}^N$ are real rooted. 

Moreover $G_n \prec A_n$, $G_n \prec G_{n+1}$ and  $A_n \prec A_{n+1}$ for each $n$. 
\end{theorem}

\section{\texorpdfstring{$\gamma$-Chow polynomials of matrices}{gamma-Chow polynomials of matrices}}\label{section gamma}
The linear space of all palindromic polynomials with center of symmetry $n/2$ has a basis $\{t^k(1+t)^{n-2k}\}_{k=0}^{\lfloor n/2\rfloor}$. Hence, if $\I_n(f)=f$, we  may define its \emph{$\gamma$-polynomial}, $\gamma(f)$, to be the unique polynomial of degree at most $\lfloor n/2\rfloor$ for which 
\[
f = (1+t)^n\gamma(f)\left(\frac{t}{(1+t)^2} \right),
\]
see \cite{athanasiadis-gamma-positivity}. 
These polynomials satisfy the following properties that we will use freely. 
\begin{itemize}
    \item $\gamma(fg) = \gamma(f)\gamma(g)$,
    \item in particular, $\gamma((1+t)f) = \gamma(f)$ and $\gamma(tf) = t\gamma(f)$,
    \item if $f$ and $g$ have the same center of symmetry, $\gamma(f+g) = \gamma(f) + \gamma(g)$.
\end{itemize}

\subsection{\texorpdfstring{Coefficients of $\gamma$-Chow polynomials of matrices}{Coefficients of gamma-Chow polynomials of matrices}}
Given a lower triangular matrix $R=(r_{n,k})_{n,k=0}^N$ with all diagonal entries equal to one, we know that the Chow polynomials  $H_n=H_n[R]$  are palindromic with center of symmetry  $(n-1)/2$.
We write $\gamma_n[R] = \gamma(H_n[R])$ for the $\gamma$-polynomial of $H_n[R]$, and we call it a \emph{$\gamma$-Chow polynomials} of the matrix $R$. The main goal of this section is to give an interpretation to the coefficients of these polynomials.
For a set ${S =\{s_1< \cdots < s_k\} \subseteq [n-1]}$ define 
\[
\beta_{[0,n]}(S) =  \sum_{T \subseteq S} (-1)^{|S\setminus T|} \alpha_{[0,n]}(T).
\]
A set $S$ of integers is called \emph{stable} if it does not contain any two consecutive integers.

\begin{theorem}\label{gamma-beta}
    Let $R=(r_{n,k})_{n,k=0}^N$ be a lower triangular matrix with all diagonal entries equal to one. Then
    \[
    \gamma_n[R] = \sum_{\substack{S\subseteq [n-1] \\ S\cup\{0\} \text{ stable}}} \beta_{[0,n]}(S)t^{|S|}, \quad 0\leq n \leq N.
    \]
\end{theorem}
\begin{proof}
    Let $W_{j} = \gamma\left(({t^{j+1}-1})/{(t-1)}\right)$, with the convention that $W_{-1} = 0$. Then, by Definition \ref{def:HR canonical decomposition}, we may write
    \[
    \gamma_{n}[R] = \sum_{m\geq 0}\sum_{\substack{S\subseteq [n-1] \\ |S| = m }}t^m\alpha_{[0,n]}(S)\left(\prod_{i=1}^mW_{s_i - s_{i-1} - 2}\right)W_{n-s_m - 1},
    \]
    as all the polynomials in the sum have the same center of symmetry. Notice that since $W_{-1} = 0$, we can restrict the sum to stable sets. Then
    \begin{align*}
    \gamma_n[R] &= \sum_{\substack{T = \{s_1 <\cdots < s_m\}\subseteq [n-1] \\T\cup\{0\} \text{ stable}}}(-1)^{|T|}(-t)^{|T|}\alpha_{[0,n]}(T) \cdot \left(
W_{n - j_m - 1}\ \prod_{i=1}^m W_{j_i - j_{i-1} - 1} \right) \\
    &= \sum_{\substack{T = \{j_1 <\cdots < j_m\} \subseteq [n-1] \\T\cup\{0\} \text{ stable}}}(-1)^{|T|}(-t)^{|T|}\alpha_{[0,n]}(T) \sum_{\substack{S\cup\{0\}\text{ stable} \\ [n-1] \supseteq S \supseteq T}}(-t)^{|S\setminus T|} \\
    &= \sum_{S\cup\{0\} \text{ stable}}t^{|S|} \sum_{T\subseteq S}(-1)^{|S\setminus T|}\alpha_{[0,n]}(T) \\
    &= \sum_{S\cup\{0\} \text{ stable}}t^{|S|} \beta_{[0,n]}(S),
    \end{align*}
    where in the second equality we used \cite[Lemma 4.23]{ferroni-matherne-vecchi}. In the third equality we exchanged the two sums and used that if $S\cup\{0\}$ is stable, then all of its subsets are also stable.
\end{proof}
For $C,D \subseteq [0, N]$, denote by $R[C,D]$ the submatrix of $R$ with rows  indexed by $C$ and columns indexed by $D$.
\begin{corollary}\label{cor: coefficients of gamma as determinants}
    Let $R=(r_{n,k})_{n,k=0}^N$ be a lower triangular matrix with all diagonal elements equal to one. Then  the $n$th $\gamma$-Chow polynomial is equal to 
    \[
    \gamma_n[R] = \sum_{\substack{ S \subseteq [n-1] \\ S\cup\{0\}   \text{ stable}}}\det(R[S\cup \{n\},\{0\}\cup S])t^{|S|}.
    \]
    Moreover the $\gamma$-polynomial of the $n$th augmented Chow polynomial of $R$ is equal to 
        \[
    \gamma(G_n[R]) = \sum_{\substack{ S \subseteq [n-1] \\ S \text{ stable}}}\det(R[S\cup \{n\},\{0\}\cup S])t^{|S|}.
    \]
\end{corollary}
\begin{proof}
    The proof of the first statement follows from Theorem \ref{gamma-beta} and \cite[Theorem 5.12]{branden-saud-1},  where it was proved that $\beta_{[0,n]}(S)= \det(R[S\cup \{n\},\{0\}\cup S])$ for each $S \subseteq [n-1]$.

    The proof of the second statement follows from that of the first by using \eqref{aug-matrx}. 
\end{proof}

Notice that when $R$ is $\TN$, then Corollary \ref{cor: coefficients of gamma as determinants} immediately implies that the coefficients of the $\gamma$-Chow polynomials associated to $R$ are nonnegative.

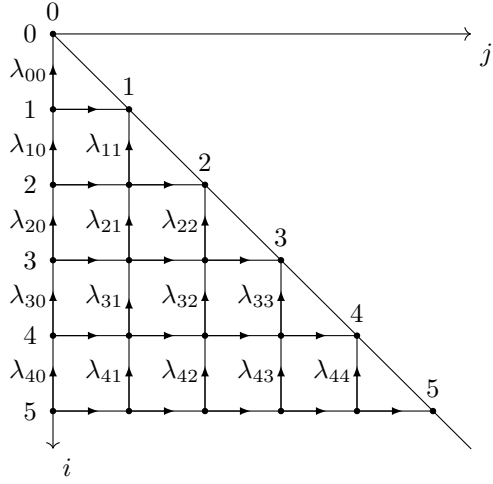
\begin{figure}
\centering
\begin{tikzpicture}[line cap=round,line join=round,x=1cm,y=1cm,scale=1]
    \draw [xstep=0.5cm,ystep=0.5cm];
    \clip(-1,-6) rectangle (7,1);
    \draw [->] (0,0) -- (5.5,0);
    \draw [->] (0,0) -- (0,-5.5);
    \draw  (0,-1) -- (1,-1);
    \draw  (0,-2) -- (2,-2);
    \draw  (0,-3) -- (3,-3);
    \draw  (0,-4) -- (4,-4);
    \draw  (0,-5) -- (5,-5);
    \draw  (1,-5) -- (1,-1);
    \draw  (2,-5) -- (2,-2);
    \draw  (3,-5) -- (3,-3);
    \draw  (4,-5) -- (4,-4);
    \draw [domain=0:5.5] plot(\x,{(-0-2*\x)/2});
    \draw [->,-latex] (0,-1) -- (0,-0.4);
    \draw [->,-latex] (0,-1) -- (0.6,-1);
    \draw [->,-latex] (0,-2) -- (0,-1.4);
    \draw [->,-latex] (0,-2) -- (0.6,-2);
    \draw [->,-latex] (1,-2) -- (1,-1.4);
    \draw [->,-latex] (1,-2) -- (1.6,-2);
    \draw [->,-latex] (0,-3) -- (0,-2.4);
    \draw [->,-latex] (0,-3) -- (0.6,-3);
    \draw [->,-latex] (1,-3) -- (1,-2.4);
    \draw [->,-latex] (1,-3) -- (1.6,-3);
    \draw [->,-latex] (2,-3) -- (2,-2.4);
    \draw [->,-latex] (2,-3) -- (2.6,-3);
    \draw [->,-latex] (0,-4) -- (0,-3.4);
    \draw [->,-latex] (0,-5) -- (0,-4.4);
    \draw [->,-latex] (0,-4) -- (0.6,-4);
    \draw [->,-latex] (0,-5) -- (0.6,-5);
    \draw [->,-latex] (1,-4) -- (1,-3.5);
    \draw [->,-latex] (1,-4) -- (1.6,-4);
    \draw [->,-latex] (2,-4) -- (2,-3.4);
    \draw [->,-latex] (2,-4) -- (2.6,-4);
    \draw [->,-latex] (3,-4) -- (3,-3.4);
    \draw [->,-latex] (3,-4) -- (3.6,-4);
    \draw [->,-latex] (1,-5) -- (1,-4.4);
    \draw [->,-latex] (1,-5) -- (1.6,-5);
    \draw [->,-latex] (2,-5) -- (2,-4.4);
    \draw [->,-latex] (2,-5) -- (2.6,-5);
    \draw [->,-latex] (3,-5) -- (3,-4.4);
    \draw [->,-latex] (3,-5) -- (3.6,-5);
    \draw [->,-latex] (4,-5) -- (4,-4.4);
    \draw [->,-latex] (4,-5) -- (4.6,-5);
    \draw (0,0.3) node {0};
    \draw (1,-0.7) node {1};
    \draw (2,-1.7) node {2};
    \draw (3,-2.7) node {3};
    \draw (4,-3.7) node {4};
    \draw (5,-4.7) node {5};
    \draw (-0.3,0) node {0};
    \draw (-0.3,-1) node {1};
    \draw (-0.3,-2) node {2};
    \draw (-0.3,-3) node {3};
    \draw (-0.3,-4) node {4};
    \draw (-0.3,-5) node {5};
    \draw (-0.7,-0.2) node[anchor=north west] {$\lambda_{00}$};
    \draw (-0.7,-1.2) node[anchor=north west] {$\lambda_{10}$};
    \draw (-0.7,-2.2) node[anchor=north west] {$\lambda_{20}$};
    \draw (-0.7,-3.2) node[anchor=north west] {$\lambda_{30}$};
    \draw (-0.7,-4.2) node[anchor=north west] {$\lambda_{40}$};
    \draw (0.3,-1.2) node[anchor=north west] {$\lambda_{11}$};
    \draw (0.3,-2.2) node[anchor=north west] {$\lambda_{21}$};
    \draw (0.3,-3.2) node[anchor=north west] {$\lambda_{31}$};
    \draw (0.3,-4.2) node[anchor=north west] {$\lambda_{41}$};
    \draw (1.3,-4.2) node[anchor=north west] {$\lambda_{42}$};
    \draw (2.3,-4.2) node[anchor=north west] {$\lambda_{43}$};
    \draw (3.3,-4.2) node[anchor=north west] {$\lambda_{44}$};
    \draw (1.3,-3.2) node[anchor=north west] {$\lambda_{32}$};
    \draw (2.3,-3.2) node[anchor=north west] {$\lambda_{33}$};
    \draw (1.3,-2.2) node[anchor=north west] {$\lambda_{22}$};
    \draw (5.5,0) node[anchor=north west] {$j$};
    \draw (0,-5.5) node[anchor=north west] {$i$};
    \begin{scriptsize}
        \draw [fill=black] (0,0) circle (1.0pt);
        \draw [fill=black] (2,-2) circle (1.0pt);
        \draw [fill=black] (0,-1) circle (1.0pt);
        \draw [fill=black] (0,-2) circle (1.0pt);
        \draw [fill=black] (1,-2) circle (1.0pt);
        \draw [fill=black] (0,-3) circle (1.0pt);
        \draw [fill=black] (1,-3) circle (1.0pt);
        \draw [fill=black] (2,-3) circle (1.0pt);
        \draw [fill=black] (0,-4) circle (1.0pt);
        \draw [fill=black] (0,-5) circle (1.0pt);
        \draw [fill=black] (1,-4) circle (1.0pt);
        \draw [fill=black] (2,-4) circle (1.0pt);
        \draw [fill=black] (3,-4) circle (1.0pt);
        \draw [fill=black] (1,-5) circle (1.0pt);
        \draw [fill=black] (2,-5) circle (1.0pt);
        \draw [fill=black] (3,-5) circle (1.0pt);
        \draw [fill=black] (4,-5) circle (1.0pt);
        \draw [fill=black] (1,-1) circle (1.0pt);
        \draw [fill=black] (3,-3) circle (1.0pt);
        \draw [fill=black] (4,-4) circle (1.0pt);
        \draw [fill=black] (5,-5) circle (1.0pt);
    \end{scriptsize}
\end{tikzpicture}
\caption{The directed graph $\Gamma_N$ with its weights.}
\label{fig:Gamma}
\end{figure}

 Suppose $R$ is a resolvable matrix. Then the minors of $R$ have a combinatorial interpretation in terms  of the $\lambda_{n,k}$'s as follows. 
 For $N \in \NN\cup\{\infty\}$, let $\Gamma_N$ be the directed graph on $\{ (i,j) \in \NN^2: j\leq i \leq N\}$ with edges 
$$
(i,j) \to (i,j+1) \ \ \ \mbox{ and } \ \ \ (i+1,j) \to (i,j), 
$$
and let $\lambda=(\lambda_{i,j})_{0\leq j \leq i < N}$ be an array of nonnegative numbers.

Attach the weight $\lambda_{i,j}$ to the vertical edge $(i+1,j) \to (i,j)$, and attach the weight $1$ to each horizontal edge, see Figure \ref{fig:Gamma}. Let further $r_{n,k}(\Gamma_N, \lambda)$ be the weighted sum of all paths from $(n,0)$ to $(k,k)$, where the weight of a path is the product of the weights of the edges used in the path, and let $R(\Gamma_N, \lambda)= \left(r_{n,k}(\Gamma_N, \lambda)\right)_{n,k=0}^N$.  By \cite[Theorems 2.1 and 2.6]{branden-saud-1}, $R=R(\Gamma_N, \lambda)$.

 Associate to $\sigma \in \mathfrak{S}_n$ a function $f=f_\sigma : [n] \rightarrow \NN$ for which $f(i)\leq i-1$ for each $i$ by  
$$
f_\sigma(j)= |\{i<j:\sigma(i)>\sigma(j)\}|. 
$$ 
Recall that $i \in [n-1]$ is called a \emph{descent} of $\sigma \in \mathfrak{S}_n$ if $\sigma(i) > \sigma(i+1)$. Let 
$\mathrm{D}(\sigma)= \{i \in [n-1] :  \sigma(i) > \sigma(i+1)\}$ denote the \emph{descent set} of $\sigma$. 
The following theorem now follows from a result due to Gessel and Viennot \cite[Corollary 6]{gessel-viennot}. 
\begin{theorem}\label{gessel-viennot-cor}
Let $R$ be a resolvable matrix, and let $\{\lambda_{n,k}\}$ be the associated array of nonnegative numbers, see Definition \ref{resolv}. Then 
$$
\det(R[S\cup\{n\}, \{0\}\cup S])= \sum_{\substack{ \sigma \in \mathfrak{S}_n  \\ \mathrm{D}(\sigma)=S}} \prod_{i=1}^n \lambda_{i-1, f_\sigma(i)}. 
$$
\end{theorem}
The next theorem now follows from Theorems \ref{gamma-beta} and \ref{gessel-viennot-cor}. 

\begin{theorem}\label{gamma-lambdas}
  Let $R$ be a resolvable matrix, and let $\{\lambda_{n,k}\}$ be the associated array of nonnegative numbers. Then 
\begin{align*}
      \gamma_n[R] &= \sum_{\substack{ \sigma \in  \mathfrak{S}_n  \\ \mathrm{D}(\sigma)\cup\{0\} \text{ stable} }} t^{\des(\sigma)} \prod_{i=1}^n \lambda_{i-1, f_\sigma(i)}, \quad \mbox{ and } \\
      \gamma(G_n[R]) &= \sum_{\substack{ \sigma \in  \mathfrak{S}_n  \\ \mathrm{D}(\sigma) \text{ stable} }} t^{\des(\sigma)} \prod_{i=1}^n \lambda_{i-1, f_\sigma(i)}.
\end{align*}

\end{theorem}

\subsection{\texorpdfstring{Interlacing properties of $\gamma$-Chow polynomials}{Interlacing gamma-Chow polynomials}}
Next we will prove some refined interlacing properties satisfied by $\gamma$-Chow polynomials of $\TN$-matrices. These will be used in Section~\ref{pavingsection} to prove that Chow polynomials of paving matroids are real-rooted. 

The following result relates the zeros of a polynomial with the zeros of its $\gamma$-polynomial.
\begin{lemma}\label{gammaint}
Suppose $g$ and $h$ are palindromic polynomials in $\RR_{\geq 0}[t]$. Then $g$ is real-rooted if and only if $\gamma(g)$ is real-rooted. Moreover: 
\begin{enumerate}
    \item If $g$ and $h$ have center of symmetry $n/2$ and $(n+1)/2$, respectively, then $g \prec h$ if and only if $\gamma(g) \prec \gamma(h)$.
    \item If $g$ and $h$ have the same center of symmetry, then $g \prec (1+t)h$ if and only if $\gamma(g) \prec \gamma(h)$.
\end{enumerate}
\end{lemma}

\begin{proof}
 The initial statement is well known and straightforward. For (1), we may use Lemma \ref{addz} and   $\gamma(tf) = t\gamma(f)$ to reduce it to the case when $0$ is not a zero of $g$ or $h$. Then (1) follows as in \cite[Proposition~2.5]{hosterstump}. Statement (2) follows from (1) using $\gamma((1+t)f) = \gamma(f)$. 
\end{proof}

Let $R$ be a lower triangular matrix with all diagonal entries equal to one. Define  polynomials 
$\sigma_{n,k}$ and $\tau_{n,k}$, $0\leq k \leq n$,  in $\R[t]$ by 
$$
\sigma_{n,k} = \gamma(\SSS_n(d_{n,k})) \quad \mbox{ and } \quad \tau_{n,k} = \gamma(\SSS_{n+1}(d_{n,k})). 
$$

We collect some simple properties of the operators $\SSS_n$ in a lemma. 
\begin{lemma}\label{snid}
For any polynomial $f$ of degree at most $n$, 
\begin{align*}
\SSS_{n+1}(f) &= t\SSS_n(f) + f, \ \ \SSS_n^2(f)= \SSS_n(f),  \\
\I_n\SSS_{n+1}(f)& = \SSS_{n+1}(f) \ \     \mbox{ and }  \ \ \SSS_{n+1}\SSS_n(f) = (1+t) \SSS_n(f). 
\end{align*}
\end{lemma}

\begin{lemma}\label{recsys}
For $0\leq k \leq n+1$, 
\begin{align*}
\sigma_{n+1, k} &= \sum_{j=k}^n \lambda_{n,j}\tau_{n,j}, \\
\tau_{n+1,k} &= t\sum_{j=0}^{k-1}\lambda_{n,j}\sigma_{n,j} +\sum_{j=k}^n \lambda_{n,j}\tau_{n,j}. 
\end{align*}
\end{lemma}

\begin{proof}
By Lemmas \ref{dnkrec} and \ref{snid}, 
\begin{equation}\label{sn1eq}
d_{n+1,k}= \SSS_{n+1}\left(\sum_{j \geq 0}\lambda_{n,j}d_{n,j}\right) - \sum_{j <k}\lambda_{n,j}d_{n,j}.
\end{equation}
Hence, by Lemma  \ref{snid}, 
$$
\SSS_{n+1}(d_{n+1,k})= \SSS_{n+1}\left(\sum_{j \geq 0}\lambda_{n,j}d_{n,j}\right) - \SSS_{n+1}\left(\sum_{j <k}\lambda_{n,j}d_{n,j}\right), 
$$
which proves the first equation by applying $\gamma$ on both sides. 

Applying $\SSS_{n+2}$ on both sides of \eqref{sn1eq}, using Lemma \ref{snid} yields after some computations
$$
\SSS_{n+2}(d_{n+1,k}) = (1+t)\SSS_{n+1}\left(\sum_{j \geq k}\lambda_{n,j}d_{n,j}\right)+ t \SSS_n \left(\sum_{j<k}\lambda_{n,j}d_{n,j}\right),
$$
from which the second equation follows after applying $\gamma$ on both sides. 
\end{proof}

\begin{theorem}\label{gammasysint}
If $R$ is resolvable and $0\leq n \leq N$, then 
\begin{enumerate}
\item $\sigma_{n,i} \prec \tau_{n,j}$ for all $i<j$, and 
\item $\{\sigma_{n,j}\}_{j=0}^n$ and $\{\tau_{n,j}\}_{j=0}^n$ are interlacing sequences. 
\end{enumerate}
\end{theorem}

\begin{proof}
By Theorem \ref{mainman}, Lemma \ref{snid} and  Lemma \ref{daint}, 
$$
\SSS_n(d_{n,i}) \prec t \SSS_n(d_{n,j})+d_{n,j}= \SSS_{n+1}(d_{n,j}), 
$$
which implies (1) by Lemmas  \ref{snid} and \ref{gammaint}. 

The proof of (2) is by induction over $n$, the case when $n=0$ being trivial. 

Assume true for $n$. Then $\{\lambda_{n,j}\tau_{n,j}\}_{j=0}^n$ is an interlacing sequence, and hence so is $\{\sigma_{n+1,j}\}_{j=0}^{n+1}$ by e.g. \cite[Corollary 7.8.6]{branden}. Notice that 
$$
\tau_{n+1,0}= \gamma((1+t)d_{n+1,0}) \mbox{ and } \tau_{n+1, n+1}=\gamma(d_{n+1,n+1}).
$$
Hence by Lemma \ref{gammaint}, $\tau_{n+1,0} \prec \tau_{n+1, n+1}$ if and only if 
$
(1+t)d_{n+1,0} \prec (1+t)d_{n+1,n+1}
$,
which is true by Theorem \ref{mainman}. By Lemma \ref{wagner}, it remains to prove $\tau_{n+1,k}\prec \tau_{n+1,k+1}$. To this end, let 
$$
f_0= \sum_{j=0}^{k-1}\lambda_{n,j}\sigma_{n,j}, \ \ f_1= \lambda_{n,k}\sigma_{n,k}, \ \ f_2 = \lambda_{n,k}\tau_{n,k} \ \mbox{ and } \ f_3= \sum_{j=k+1}^n \lambda_{n,j}\tau_{n,j}. 
$$
Since $\{\sigma_{n,j}\}_{j=0}^n$ and  $\{\tau_{n,j}\}_{j=0}^n$ are interlacing, it follows using (1) and Proposition~\ref{cones} that the sequence $f_0, f_1, f_2, f_3$ is interlacing. By construction, 
$$
\begin{pmatrix}
\tau_{n+1,k}\\
\tau_{n+1,k+1}
\end{pmatrix}
=
\begin{pmatrix}
t & 0 & 1 & 1\\
t & t & 0 & 1
\end{pmatrix}
\begin{pmatrix}
f_0\\
f_1\\
f_2\\
f_3
\end{pmatrix}. 
$$
It is known that this matrix preserves interlacing, see e.g. \cite[Theorem 2.2]{Zhang}. This concludes the proof.  
\end{proof}
The following corollary will be used to prove that Chow polynomials of paving matroids are real-rooted. 
\begin{corollary}\label{forthepave}
If $R$ is resolvable and $0\leq n<N$, then 
$$
\gamma(d_{n,0}) \prec \gamma(d_{n+1,n+1}/t). 
$$
\end{corollary}
\begin{proof}
By Lemma \ref{recsys},
$$
\gamma(d_{n,0}) = \sigma_{n,0} \ \ \mbox{ and } \ \ \gamma(d_{n+1,n+1}/t)= \tau_{n+1,n+1}/t=\sum_{j=0}^{n}\lambda_{n,j}\sigma_{n,j}. 
$$
The proof now follows from Theorem \ref{gammasysint} and Proposition \ref{cones}. 
\end{proof}

\section{\texorpdfstring{Chow polynomials of $\mathrm{TN}$-posets and dual $\mathrm{TN}$-posets}{Chow polynomials of TN-posets and dual TN-posets}}\label{section real zeros TN posets}
In this paper, a weakly rank-uniform poset $P$ is called a $\TN$-\emph{poset} if the matrix $R(P)$ is $\TN$. Below is a list of examples of $\TN$-posets, for proofs see \cite{branden-saud-1, branden-saud-2}. 

\begin{itemize}
\item[\textbf{a}.] Boolean cell complex with nonnegative $h$-vectors. For example  Cohen-Macaulay simplicial complexes and face lattices of simplicial polytopes. 
\item[\textbf{b}.] Cubical complexes with nonnegative cubical $h$-vectors \cite{adin1996new}. These include face lattices of cubical polytopes, i.e., polytopes for which all faces are hypercubes. 
\item[\textbf{c}.] $q$-posets with nonnegative $h$-vectors \cite{alder2010q}. For example shellable $q$-complexes \cite{Ghorpade}. 
\item[\textbf{d}.] Perfect matroid designs \cite{Deza}, i.e., rank uniform geometric lattices. These include (truncations of) projective geometries and affine geometries. 
\item[\textbf{e}.] Dual of Dowling lattices \cite{dowlingGroups}. In particular, the dual of partition lattices. 
    \end{itemize}
Recall that if $P$ is a weakly ranked poset of rank $r$ and $S \subseteq [r-1]$, then 
$$
P_S= \{ x \in P : \rho(x) \in S \cup \{0,r\}\}
$$
is the \emph{rank selected subposet} induced by $S$. Since submatrices of $\TN$-matrices are $\TN$ it follows that the class of $\TN$-posets is closed under rank selection. Hence any rank-selected poset of any poset listed above is $\TN$. 
\begin{theorem}\label{mainposet}
Suppose $P$ is a $\TN$-poset, or the dual of a $\TN$-poset, of rank $n$. Then the Chow polynomials $H_P, d_P, G_P$ and $A_P$ are real-rooted. 

Moreover if $P$ is $\TN$, then  $H_{\tau(P)} \prec H_P$ and $H_{[\zero, x]}\prec H_P$ for any $x$ of rank $n-1$. 
\end{theorem}

\begin{proof}
Suppose $P$ is a $\TN$-poset. Let $R=R(P)$. By Proposition \ref{matrix-poset}, $H_n=H_P$,  $d_n=t H_{\tau(P)}$ and $H_{n-1}=H_{[\zero, x]}$. Hence the statements for $\TN$-posets follow from Theorems \ref{mainman2} and \ref{mainman2A}. 

The case when $P^*$ is $\TN$ follows from the above and Theorem \ref{dualfunc}, noting again that $d_P=t H_{\tau(P)}$ and that $\tau(P)$ is a dual $\TN$-poset. 
\end{proof}
If we apply Theorem \ref{mainposet} to \textbf{a} above, then we recover the result of Hoster and Stump \cite{hosterstump}. Theorem \ref{mainposet} for the case \textbf{e} above implies in particular that Chow polynomials of Dowling lattices are real-rooted\footnote{Coron, Ferroni, and Li informed us that they independently found a different proof for the real-rootedness of Chow polynomials of Dowling lattices.}.
\begin{lemma}\label{trvshyper}
Suppose $P$ is a $\TN$-poset of rank $n$, and let $x$ be an element of rank $n-1$. Then 
$$
\gamma(H_{[\zero, x]}) \prec \gamma(H_{\tau(P)}). 
$$
\end{lemma}

\begin{proof}
The lemma follows immediately from Corollary \ref{forthepave}, given Proposition \ref{matrix-poset}. 
\end{proof}

\begin{corollary}\label{H interlaces G}
Let $P$ be a dual $\TN$-poset. Then the characteristic Chow polynomial interlaces the characteristic augmented Chow polynomial, i.e., $H_{P} \prec G_{P}$. 
\end{corollary}
\begin{proof}
    Let $P^*$ denote the poset dual to $P$. By Theorems \ref{mainman2A} and \ref{dualfunc} we deduce $$G_P = G_{P^*} \prec A_{P^*} =tH_P,$$
    from which the result follows.
\end{proof}

\section{Chow polynomials of lattices of flats of paving matroids}\label{pavingsection}
Let $P$ be a graded, rank uniform and bounded poset of rank $n$ and let $0\leq d <n$. Consider a subset $\mathcal H \subseteq P$ of pairwise incomparable elements such that $d < \rho(x) < n$ for each $x \in \mathcal H$. Let $P(d,\mathcal H)$ be the weakly ranked poset of rank $n$ obtained by adjoining $\one$ and $\mathcal H$ to the set $\{x \in P \mid \rho(x) \leq d\}$ and by declaring all elements in $\mathcal H$ to be of rank $d+1$ and $\rho(\one) = d+2$. Notice also that if 
\begin{equation}\label{paving graded condition}
    \text{for each $x$ of rank $\rho(x) \leq d$ there exists a $y \in \mathcal H$ such that $x\leq y$}, 
\end{equation}
then $P(d,\mathcal H)$ is graded again.
If $P$ is a boolean algebra on $n$ elements and $\mathcal H$ satisfies  \eqref{paving graded condition}, then $P(d,\mathcal H)$ is a paving geometric lattice of rank $d+2$.
\begin{theorem}\label{thm:main paving}
    Let $P$ be a graded $\TN$-poset of rank $n$ and $P' = P(d,\mathcal H)$ be as above. Then the  polynomials $d_{P'}, H_{P'}, A_{P'}$ and $G_{P'}$ are real-rooted.
\end{theorem}
\begin{proof}
  By truncations and  Theorem \ref{thm: G is H, A is d} it suffices to prove that $H_{P'}$ is real-rooted. Notice that Corollary \ref{cor: H alternative truncation formula} implies the identity 
    \[
    \gamma_{P'}(t) = \gamma_{\tau(P')}(t) - t\gamma_{\tau^2(P')}(t) + t\sum_{y \in \mathcal H}\gamma_{\tau[\zero,y]}(t).
    \]
   By construction all posets involved in the right-hand side are $\TN$. By Theorem \ref{mainposet} and Lemma \ref{gammaint} we know that $\gamma_{\tau^2(P')}(t) \prec \gamma_{\tau(P')}(t)$, which implies $-t\gamma_{\tau^2(P')}(t) \prec \gamma_{\tau(P')}(t)$. 

   We claim that $\gamma_{\tau[\zero,y]}(t) \prec \gamma_{\tau^2(P')}(t)$ for each $y \in \mathcal H$. This claim implies that ${-t\gamma_{\tau^2(P')}(t) \prec t\gamma_{\tau[\zero,y]}(t)}$, and hence by Proposition \ref{cones} and  $-t\gamma_{\tau^2(P')}(t) \prec \gamma_{\tau(P')}(t)$ we may deduce $-t\gamma_{\tau^2(P')}(t) \prec \gamma_{P'}(t)$, from which real-rootedness follows.
   
It remains to prove the claim. Consider the rank selected subposet $P_S$, where $S=\{1,2\ldots, d-1, \rho(y)\}$.  Then the subposet $\tau[\zero,y]$ of $P'$ is equal to the interval $[\zero, y]$ in $P_S$. Also, $\tau^2(P')$ is equal to $\tau(P_S)$. Since $P_S$ is $\TN$, the claim now follows from Lemma \ref{trvshyper}.
\end{proof}
\begin{corollary}\label{cor:main paving}
    If $L$ is the lattice of flats of a  paving matroid, then the polynomials the  polynomials $d_{L}, H_{L}, A_{L}$ and $G_{L}$ are real-rooted. 
\end{corollary}

\section{Toeplitz matrices}\label{section toeplitz}
In this section we will focus on the special case when $R$ is a Toeplitz matrix. 
Let $\{a_n\}_{n=0}^\infty$ be a sequence of elements in an integral domain $\R$, where $a_0=1$.  Associate to $\{a_n\}_{n=0}^\infty$  the lower triangular \emph{Toeplitz matrix}
$R=(a_{n-k})_{n,k=0}^\infty$, 
where $a_m=0$ if $m<0$. 

Consider the formal power series 
$f= \sum_{n=0}^\infty a_nz^n$ in $\R[[z]]$, and define formal power series in $\R[t][[z]]$ by 
\begin{alignat*}{2}
 D(z,t) &= \sum_{n =0}^\infty d_n(t) z^n, \quad 
 H(z,t) &= \sum_{n=0}^\infty H_n(t) z^n, \\
 A(z,t) &= \sum_{n =0}^\infty A_n(t) z^n, \quad 
  G(z,t) &= \sum_{n=0}^\infty G_n(t) z^n. 
  \end{alignat*}
where $d_n(t), H_n(t),  A_n(t)$ and $G_n(t)$  are the Chow-derangement polynomials, Chow polynomials, Chow-Eulerian polynomials  and augmented Chow polynomials associated to $R$, respectively. 

\begin{theorem}\label{fourform}
Let $f= \sum_{n=0}^\infty a_nz^n$ be a formal power series in $\R[[z]]$, where $a_0=1$. Then 
\begin{align*}
D(z,t) &= \sum_{n =0}^\infty d_n(t) z^n= \frac {1-t}{f(tz)-tf(z)}=  \frac  1 {1- \sum_{n \geq 2 } a_n(t+t^2+\cdots +t^{n-1})z^n},  \\
H(z,t) &= \sum_{n=0}^\infty H_n(t) z^n= \frac {(1-t)f(z)}{f(tz)-tf(z)} =\frac {1-t} { \frac {f(tz)} {f(z)}-t}, \\
A(z,t) &= \sum_{n =0}^\infty A_n(t) z^n= \frac {(1-t)f(tz)}{f(tz)-tf(z)}= 1-t+tH(z,t),  \\
G(z,t) &= \sum_{n=0}^\infty G_n(t) z^n= \frac {(1-t)f(tz)f(z)}{f(tz)-tf(z)}.
\end{align*}
\end{theorem}

\begin{proof}
Since $\I_{n}(H_n)= tH_n$, for $n \geq 1$, we deduce from Corollary \ref{cor: H and d for matrices}, 
$$
\sum_{k=0}^n a_{n-k} t^n d_k(1/t) =t \sum_{k=0}^n a_{n-k}  d_k(t), \quad n \geq 1.
$$
Since $t^kd_k(1/t)= d_k$ for all $k$, the above equation reduces to 
$$
\sum_{k=0}^n a_{n-k} t^{n-k} d_k =t \sum_{k=0}^n a_{n-k}  d_k, \quad n \geq 1.
$$
The left hand side is the coefficient of $z^n$ in $f(tz)D(z,t)$, while the right hand side is the coefficient of $z^n$ in $tf(z)D(z,t)$, from which we deduce 
$$
D(z,t)f(tz)= 1-t +tf(z)D(z,t),
$$
which yields the first identity. The second identity follows immediately from the first combined Corollary \ref{cor: H and d for matrices}. The third and fourth identities follow similarly using Corollary \ref{cor: G and A for matrices}. 
\end{proof}

The next proposition provides the generating function for the Chow polynomials of iterated truncations of the Toeplitz matrix corresponding to $f(z)$. 
 \begin{proposition}\label{trunc-ident}
Let $R$  be the Toeplitz matrix corresponding to $f= \sum_{n=0}^\infty a_n z^n$. For  $n \geq 0$ and $k \geq 1$, let  $R(n,k)$ denote the matrix obtained from $R$ by deleting all rows and columns indexed by $n,n+1, \ldots, n+k-1$.  Then 
\begin{align*}
    \sum_{n \geq 0} d_n[R(n,k)](t) z^n &= 1+  zt\frac {f_{k+1}(z)-f_{k+1}(tz)} {f(tz)-tf(z)}, \\
    \sum_{n \geq 0} H_n[R(n,k)](t) z^n &= 1+  \frac {f_k(z)-f_k(tz)} {f(tz)-tf(z)}, \\
    \sum_{n \geq 0} A_n[R(n,k)](t) z^n &= 1+  \frac {tf_{k}(z)-f_{k}(tz)} {f(tz)-tf(z)}f(tz), \\
    \sum_{n \geq 0} G_n[R(n,k)](t) z^n &= 1+  \frac {f_k(z)-f_k(tz)} {f(tz)-tf(z)}f(tz),
 \end{align*}
 where $f_k(z)= \sum_{j =0}^{\infty} a_{k+j} z^j$. 
 \end{proposition} 

 \begin{proof}
 We prove the second identity. The others follow similarly. Let $d_n=d_n[R]$. Using Corollary \ref{cor: H and d for matrices}, we compute 
\begin{align*}
H_n[R(n,k)] &= d_n[R(n,k)] + \sum_{j=0}^{n-1} r_{n+k,j} d_j 
          = t\SSS_{n-1}\left(\sum_{j=0}^{n-1} a_{n+k-j} d_j \right) + \sum_{j=0}^{n-1} a_{n+k-j} d_j \\
          &= \sum_{j=0}^{n-1} a_{k+n-j} \frac{t^{n-j}-1}{t-1}d_j 
          = \sum_{j=0}^{n} b_{n-j}(t) d_j,  
\end{align*}
where 
$$
\sum_{n \geq 0}b_n(t) z^n = \frac {f_k(tz)- f_k(z)}{t-1}. 
$$
The proof now follows from the above expression for the series $D(z,t)$. 
 \end{proof}

A sequence $\{a_n\}_{n=0}^\infty$ of real numbers is a \emph{P\'olya frequency sequence} if the Toeplitz matrix $R=(a_{n-k})_{n,k=0}^\infty$ is $\TN$. 
P\'olya frequency sequences were characterized by Aissen, Schoenberg,  Whitney and Edrei \cite{ASW} as follows. 

\begin{theorem}\label{ASWE}
    A sequence $\{a_n\}_{n=0}^\infty$ of real numbers is a P\'olya frequency sequence if and only if its generating function is of the form 
    \begin{equation}
    \label{PFS}
     f=   \sum_{n=0}^\infty a_n z^n = C z^N e^{\gamma z} \prod_{i=1}^\infty \frac {1+ \alpha_iz} {1- \beta_iz},
    \end{equation}
    where $C, \gamma, \alpha_i, \beta_i$ are nonnegative real numbers, $N \in \NN$, and $\sum_{i=1}^\infty (\alpha_i+\beta_i)<\infty$. 
\end{theorem} 

Applying Theorem \ref{mainman2} to Toeplitz matrices produces four families of real-rooted polynomials to any series $f$ of the form \eqref{PFS}. 

\begin{theorem}\label{PFSR}
Suppose $f$ is a power series as in \eqref{PFS}. Then the polynomials $d_n(t), H_n(t),  A_n(t)$ and $G_n(t)$, $n\geq 0$,  defined via $f$ by the identities in Theorem~\ref{fourform} 
are all real-rooted. 
\end{theorem}

\begin{proof}
When $a_0 >0$, then the theorem follows immediately from Theorems \ref{ASWE} and \ref{mainman2}. The case when $a_0=0$, follows by continuity and Hurwitz' theorem on the continuity of zeros \cite[Theorem 1.3.8]{Scheidemann} by considering 
$$
f_\epsilon= C (\epsilon +z)^N e^{\gamma z} \prod_{i=1}^\infty \frac {1+ \alpha_iz} {1- \beta_iz},
$$
and letting $\epsilon \to 0$. 
\end{proof}

\subsection{Binomial and Sheffer posets}\label{binopo}
 Recall \cite{DRS, stanley1976binomial,stanley-ec1} that a  \emph{binomial poset} is a  locally finite poset $P$ for which 
\begin{itemize}
\item there exists an infinite chain in $P$,
\item each interval of $P$ is graded, and 
\item there exists a function $B : \NN \to \NN$, called the \emph{factorial function} of $P$, such that the number of maximal chains in any interval $[x,y]$ in $P$ is equal to $B(\rho(x,y))$. 
\end{itemize}
Hence $P$ is rank uniform with 
\begin{equation}\label{rank-B}
    r_{n,k}(P) = \frac {B(n)} {B(k)\cdot B(n-k)},  \quad  0 \leq k \leq n, 
\end{equation}
since each maximal chain in $[\zero, x]$, $\rho(x)=n$,  passes through a unique element of rank $k$. 

\begin{theorem}\label{th: bin-pos-ser}
Let $P$ be a binomial poset with factorial function $B$, and let 
$$
b(z) = \sum_{n=0}^\infty \frac {z^n}{B(n)}.
$$
The generating functions for the various Chow-polynomials of $P$ have the following expressions.
\begin{alignat*}{2}
\sum_{n =0}^\infty d_n(t) \frac {z^n}{B(n)} &= \frac {1-t}{b(tz)-tb(z)}, \quad \quad 
\sum_{n=0}^\infty H_n(t) \frac {z^n}{B(n)} &= \frac {(1-t)b(z)}{b(tz)-tb(z)},\qquad\\
\sum_{n =0}^\infty A_n(t) \frac {z^n}{B(n)} &= \frac {(1-t)b(tz)}{b(tz)-tb(z)}, \quad  \quad 
\sum_{n=0}^\infty G_n(t) \frac {z^n}{B(n)} &= \frac {(1-t)b(tz)\cdot b(z)}{b(tz)-tb(z)}.
\end{alignat*}
\end{theorem}   
    \begin{proof}
        The proof follows from \eqref{rank-B} by applying Proposition \ref{scaling} and Theorem \ref{fourform}.
    \end{proof}

Let $\mathbb{F}_q$ be a field with $q$ elements, and let $V(q)$ be the free $\mathbb{F}_q$-linear space over the set $\{e_1, e_2, \ldots\}$.  Let further 
$\mathbb{B}(q)$ be the lattice of all finite dimensional subspaces of $V(q)$.  Then $\BB(q)$ is boolean with factorial function given by 
$$\mathbf{(n)!}=1 \cdot(1+q) \cdots (1+q+\cdots + q^{n-1}).$$    
 By Theorem~\ref{th: bin-pos-ser}, 
$$
\sum_{n=0}^\infty  H_n(t) \frac {z^n}{\mathbf{(n)!}}  = \frac {(1-t)e_q(z)}{e_q(tz)-te_q(z)}, 
 $$
 where $e_q(z)= \sum_{n=0}^\infty z^n/\mathbf{(n)!}$ is the $q$-\emph{exponential function.} For $\sigma \in \mathfrak{S}_n$, let $\textrm{maj}(\sigma)= \sum_{i\in \mathrm{D}(\sigma)} i$. From the work of Shareshian and Wachs \cite{Shareshian-Wachs-10} it follows that 
 $$
 H_n(t) = \sum_{\sigma \in \mathfrak{S}_n} q^{\textrm{maj}(\sigma)-\textrm{exc}(\sigma)} t^{\textrm{exc}(\sigma)}, 
 $$
 the $q$-analog, $A_n(q,t)$, of the Eulerian polynomial studied by Shareshian and Wachs in \cite{Shareshian-Wachs-10}. 
 This was first proved by Hameister, Rao and Simpson \cite{hameister-rao-simpson}. 

 Moreover, it follows similarly from Theorem~\ref{th: bin-pos-ser} that the augmented Chow polynomials, $G_n(t)$, of $\BB(q)$ are the $q$-binomial Eulerian polynomials $\widetilde{A}_n(q,t)$ studied by Shareshian and Wachs in \cite{shareshian-wachs-2018}. The truncated projective geometries 
 $\BB_n^k(q)=\tau^k[\zero, x]$, where $x$ is an element of $\BB(q)$ of rank $n+k$ are geometric lattices. Proposition \ref{trunc-ident} provides identities for the various Chow polynomials associated to $\BB_n^k(q)$. For example, 
 $$
\sum_{n \geq 0} H_{\BB_n^k(q)}(t) \frac {z^n}{\mathbf{(n)!}} = 1+  \frac {\sum_{n \geq k}(1-t^k)\frac {z^k} {\mathbf{(n)!}}} {e_q(tz)-te_q(z)},  
 $$
which was first proved in \cite[Theorem 1.1]{hameister-rao-simpson} in an equivalent form. 
Since $\BB(q)$ is a $\TN$-poset we know that the polynomials $A_n(q,t)$, $\widetilde{A}_n(q,t)$ and $H_{\BB_n^k(q)}(t)$,  $n,k \in \NN$, are all real-rooted.

For $\sigma \in \mathfrak{S}_n$, let $\mathrm{inv}(\sigma)= |\{i<j  : \sigma(i) >\sigma(j)\}$.
The next result was first proved in \cite{shareshian-wachs-2018}.
\begin{corollary}
    Let $n \in \NN$. Then 
\begin{align*}
  \gamma_n[\BB(q)] &= \sum_{\substack{ \sigma \in  \mathfrak{S}_n  \\ \mathrm{D}(\sigma)\cup \{0\} \text{ stable} }} t^{\des(\sigma)} q^{\mathrm{inv}(\sigma)}, \quad \mbox{ and }\\ 
    \gamma(G_n[\BB(q)]) &= \sum_{\substack{ \sigma \in  \mathfrak{S}_n  \\ \mathrm{D}(\sigma) \text{ stable} }} t^{\des(\sigma)} q^{\mathrm{inv}(\sigma)}. 
  \end{align*}
\end{corollary}
\begin{proof}
 The matrix $R(\BB(q))$ is resolvable with $\lambda_{n,k}= q^{k-1}$ \cite[Proposition 7.1]{branden-saud-1}. Hence the corollary follows from Theorem \ref{gamma-lambdas}. 
\end{proof}

Ehrenborg and Readdy generalized the notion of binomial posets in \cite{Ehr-Rea}. A poset $P$ is called a \emph{Sheffer poset} if there are two functions $B : \NN \to \NN$ and $C: \NN \to \NN$ such that 
\begin{enumerate}
\item $P$ is locally finite with a least element $\zero$, and $P$ contains an infinite chain. 
\item Each interval in $P$ is graded, and hence $P$ has a rank function $\rho : P \to \NN$. 
\item The number of maximal chains in $[\zero, x]$ is equal to $C(\rho(x))$, for each $x \in P$. 
\item If $\zero < x \leq y$, then the number of maximal chains in $[x,y]$ is equal to $B(\rho(x,y))$. 
\end{enumerate}
The functions $B$ and $C$ are called the \emph{factorial functions} associated to $P$. Hence Sheffer posets are rank uniform, and binomial posets are the Sheffer posets for which $B=C$. 

\begin{example}\label{Ex: Sheffer}
The following posets are Sheffer. For more example,  see \cite{Ehr-Rea}. 
\begin{itemize}
\item[\textbf{a}.] Let $r$ be a positive integer. The infinite $r$-\emph{cubical lattice}  is the infinite direct product (with a least element $\zero$ adjoined)  
$$
\mathbf{C}_r = \{\zero\} \cup \prod_{n=1}^\infty N_r,  
$$
where $N_r$ is the poset consisting of an antichain on $r$ elements with a largest element adjoined. The factorial functions for $\mathbf{C}_r$ are $B(n)=n!$ and $C(n)= r^{n-1}(n-1)!$. 
\item[\textbf{b}.] The \emph{affine geometry} $\mathbb{A}(q)$, ordered by inclusion, may be defined as 
$$
\mathbb{A}(q)=  \{ U \setminus H : U \in \mathbb{B}(q)\}, 
$$
where $H$ is the subspace of $V(q)$ spanned by $\{e_2,e_3,\ldots\}$.  For positive integers $r$, the rank $r$ elements of  $\mathbb{A}(q)$ are precisely the non-empty sets of the form $U \setminus H$, where $U$ is an element of $\mathbb{B}(q)$ of rank $r$, see \cite[Chapter~6.2]{Oxley}.  It follows that $\mathbb{A}(q)$ is Sheffer (see \cite[Proposition~6.2.5]{Oxley}), and that the factorial functions are 
$$
B(n)= \mathbf{(n)}! \quad \mbox{ and } \quad C(n) = q^{n-1} \cdot \mathbf{(n-1)}!. 
$$
Indeed,  the maximal chains $\varnothing < x_1<\cdots <x_n$ in $\mathbb{A}(q)$ are in one-to-one correspondence with the maximal chains $(0) < x_1<\cdots <x_n$ in $\mathbb{B}(q)$ for which $x_1 \not \subseteq H$. Hence we should choose $x_1$ in $\mathbf{\binom n 1}- \mathbf{\binom {n-1} 1} = q^{n-1}$ ways, and then the chain $x_1<x_2<\cdots <x_n$ in $B(n-1)= \mathbf{(n-1)}!$ ways. 
\end{itemize}

\end{example}

If $P$ is a Sheffer poset, then 
$$
r_{n,k}(P)= \frac { C(n) } {C(k) \cdot B(n-k)}, \quad \quad  0<k \leq n, 
$$ 
see \cite{Ehr-Rea}. 
\begin{theorem}\label{th: shef-pos-ser}
Let $P$ be a Sheffer poset with factorial functions $B$ and $C$, and let 
$$
b(z) = \sum_{n=0}^\infty \frac {z^n}{B(n)} \quad \mbox{ and } \quad c(z) = \sum_{n=0}^\infty \frac {z^n}{C(n)}. 
$$
Then 
\begin{align*}
\sum_{n =0}^\infty d_n(t) \frac {z^n}{C(n)} &= 1+\frac {1-t}{b(tz)-tb(z)}-   \frac {c(tz)-tc(z)}{b(tz)-tb(z)}, \\
\sum_{n=0}^\infty H_n(t) \frac {z^n}{C(n)} &= \frac {(1-t)b(z)}{b(tz)-tb(z)}+ \frac {c(z)\cdot b(tz)-c(tz)\cdot b(z)}{b(tz)-tb(z)},\\
\sum_{n =0}^\infty A_n(t) \frac {z^n}{C(n)} &= 1+ t \frac {c(z)-c(tz)}{b(tz)-tb(z)}, \\
\sum_{n=0}^\infty G_n(t) \frac {z^n}{C(n)} &=  \frac {c(z)\cdot b(tz)-tc(tz)\cdot b(z)}{b(tz)-tb(z)}.
\end{align*}
\end{theorem}

\begin{proof}
We prove the  identity for $d_n$, the others follows similarly. Let $R'=(r'_{n,k})_{n,k=0}^\infty= (r_{n,k}\cdot C(k)/C(n))_{n,k=0}^\infty $, i.e., 
$$
r'_{n,k}= \begin{cases}
1/C(n), &\mbox{ if } k=0, \\
1/B(n-k), &\mbox{ if } 0<k\leq n, \\
0, &\mbox{ otherwise. }  
\end{cases}
$$
By Proposition \ref{scaling}, $d_n/C(n)= d_n[R']$. From \eqref{dn-recu}, we derive 
$$
t^n/C(n)-t^n/B(n) +\sum_{k=0}^n d_k \frac {t^{n-k}}{B(n-k)} =  t/C(n)-t/B(n) +t\sum_{k=0}^n d_k \frac {1}{B(n-k)},  \quad n\geq 0, 
$$
from which the proposed identity follows after some manipulations. 
\end{proof}
The next corollaries follow from  Theorem \ref{th: shef-pos-ser} and Example \ref{Ex: Sheffer}. 
\begin{corollary}
Let $\{H_n\}_{n=0}^\infty$ be the Chow polynomials for the $r$-cubical lattice, $r >0$. Then 
$$
\sum_{n\geq 0} H_{n+1}(t) \frac {z^n}{r^n n!}= \frac {e^{z/r+tz}-te^{tz/r+z}}{e^{tz}-te^z}. 
$$
\end{corollary}  
  
\begin{corollary}
Let $\{H_n\}_{n=0}^\infty$ be the Chow polynomials for the affine geometry $\mathbb{A}(q)$. Then 
$$
\sum_{n\geq 0} H_{n+1}(t) \frac {z^n}{q^n \mathbf{(n)}!}= \frac {e_q(z/q)\cdot e_q(tz)-te_q(tz/q)\cdot e_q(z)}{e_q(tz)-te_q(z)}. 
$$
\end{corollary}  

\subsection{Symmetric and supersymmetric functions}\label{symsec}
When 
 $$
 f(z) =  \prod_{n = 0}^\infty (1+x_n z) = \sum_{n \geq 0} e_n(\xx)z^n \quad \mbox{ or }  \quad f(z) =  \prod_{n = 0}^\infty (1-x_n z)^{-1} = \sum_{n \geq 0} h_n(\xx)z^n, 
 $$
where $e_n(\xx)$ and $h_n(\xx)$ are the $n$th elementary symmetric and complete homogeneous polynomials, respectively, 
 then a result due to Stanley, see \cite[Theorem~7.2]{Shareshian-Wachs-10}, provides a combinatorial interpretation of the Chow polynomial $H_n(t)$ for the matrices  $(e_{i-j}(\xx))_{i,j=0}^\infty$ and $(h_{i-j}(\xx))_{i,j=0}^\infty$. Here we work over the integral domain $\Lambda(\xx)$ of symmetric functions, see \cite{Macdonald, stanley2024enumerative}. The various Chow polynomials for these matrices have been extensively studied under different names in the literature. In particular, these polynomials play an important role in the celebrated work of Shareshian and Wachs on generalizations of Eulerian polynomials, see \cite{Shareshian-Wachs-10} and the references therein. 
 
 For $R=(e_{i-j}(\xx))_{i,j=0}^\infty$, Stanley proved 
 $$
 H_n(t) = \sum_{w \in W_n} t^{\des(w)} \prod_{i=1}^n x_{w(i)},  
 $$
 where $W_n$ is the set of \emph{Smirnov words} of length $n$, i.e., words $w : [n] \to \NN$ for which  $w(i) \neq w(i+1)$ for each $i$, and $\des(w)= |\{ i \in [n-1] : w(i)>w(i+1)\}|$. 
 
Let $Q_n$ be the set of words $w : [n] \to \ZZ \setminus \{0\}$ for which 
 $$
 w(i)= w(i+1)\ \ \ \mbox{ implies } \ \ \ w(i)<0.
 $$
For $w \in Q_n$, let $\col(w)= |\{i : w(i)=w(i+1)\}|$ denote the number of \emph{collisions} in $w$. Motivated by Theorem \ref{ASWE},  we extend Stanley's result to give a combinatorial interpretation of the Chow polynomials of essentially any P\'olya frequency sequence. In the following theorem we work over the integral domain $\ZZ[[x_1,y_1,x_2,y_2,\ldots]]$. 
 \begin{theorem}
 Let 
 $$
 f(z)=\sum_{n=0}^\infty e_n(\xx/\yy) z^n=  \prod_{i=1}^\infty \frac {1+ x_iz} {1- y_iz}.
$$
Then the Chow polynomials associated to the Toeplitz matrix $(e_{i-j}(\xx/\yy))_{i,j=0}^\infty$ are given by
 $$
 H_n(t) = \sum_{w \in Q_n} t^{\des(w)} (1+t)^{\col(w)} \prod_{i=1}^n x_{w(i)}, 
 $$
 where $x_{-i} = y_i$ for each $i>0$. 
 \end{theorem}
 
\begin{proof}
Let $f_0(z)= \prod_{i \in \ZZ \setminus \{0\}} (1+x_iz)= \prod_{i = 1}^\infty (1+x_iz)(1+y_iz)$. By Stanley's result, the  Chow polynomial $H_n^0(t)$ corresponding to $(e_{i-j}(\xx'))_{i,j=0}^\infty$, where $\xx'= (\ldots, x_{-2}, x_{-1}, x_1, x_2, \ldots)$, is the weighted generating function for the descent statistic over Smirnov words $W^0_n$ over the alphabet $\ZZ \setminus \{0\}$. 
 
Let $Q= \cup_{n\geq 0}Q_n$. Each word $v$ in $Q$ may be written uniquely as $$v= w_1^{\alpha_1} w_2^{\alpha_2}\cdots w_m^{\alpha_m},$$ where $w_1w_2\cdots w_m \in W^0_m$ is a Smirnov word, $\alpha_i>0$ for all $i$, and $\alpha_i =1$ whenever $w_i>0$.  It follows that 
 $$
 \sum_{n \geq 0} z^n \sum_{w \in Q_n} t^{\des(w)} (1+t)^{\col(w)} \prod_{i=1}^n x_{w(i)}
 $$
is obtained from 
  $$
H_0(z,t)= \sum_{n \geq 0} H_n^0(t) z^n = \sum_{n \geq 0} z^n \sum_{w \in W_n^0} t^{\des(w)}  \prod_{i=1}^n x_{w(i)}  =\frac {1-t} { \frac {f_0(tz)} {f_0(z)}-t}.
 $$
by the change of variables $x_i \longmapsto x_i$ for each $i>0$, and 
 $$
 y_i \longmapsto u_i=\frac {y_i} {1-(1+t)zy_i}, 
 $$
for each $i>0$. Since 
 $$
 \frac {1+u_itz}{1+u_iz}=  \frac {(1-y_itz)^{-1}}{(1-y_iz)^{-1}}, 
 $$
it follows that $f_0(tz)/f_0(z)$ is transformed to  $f(tz)/f(z)$ by this change of variables, which proves the theorem. 
\end{proof}

Let $\mu \subseteq \lambda$ be integer partitions. Then the \emph{supersymmetric skew Schur function}, $s_{\lambda/\mu}(\xx/\yy)$, may be defined by the \emph{Jacobi-Trudi identity} as
$$
s_{\lambda/\mu}(\xx/\yy) = \det\left( e_{\lambda'_i-\mu'_j-i+j}(\xx/\yy) \right)_{1\leq i,j\leq \ell(\lambda')},
$$
see \cite[I.3, Exercise 23]{Macdonald} and \cite{Moens}. 
These series have nonnegative integer coefficients. The series $s_{\lambda}(\xx/\yy) = s_{\lambda/0}(\xx/\yy)$ is called a \emph{supersymmetric Schur function}. They form a basis for the integral domain $\Lambda(\xx/\yy)$ of \emph{supersymmetric functions}, see \cite{Moens}. 
We will now use Corollary \ref{cor: coefficients of gamma as determinants} to give an alternative and unified proof of an unpublished result of Gessel, see \cite[Theorem 2.40]{athanasiadis-gamma-positivity}, who proved the case of Theorem \ref{gamma-symmetric-h}  for $H_n$ and $d_n$, and  Shareshian and Wachs \cite{shareshian-wachs-2018} who proved Theorem \ref{gamma-symmetric-h} for $G_n$. 

Let $f(\xx/\yy;t) \in \Lambda(\xx/\yy)[t]$ be a palindromic polynomial in $t$ with center of symmetry $n/2$ and with coefficients in $\Lambda(\xx/\yy)$. Then we may write  
$$
f(\xx/\yy;t) = \sum_{k=0}^{\lfloor n/2 \rfloor} \gamma_k(\xx/\yy) t^k (1+t)^{n-2k}.
$$
We say that $f(\xx/\yy;t)$ is \emph{super Schur $\gamma$-positive} if each $\gamma_k(\xx/\yy)$ has a nonnegative expansion in super Schur functions. Notice that $f(\xx/\yy;t) \in \Lambda(\xx/\yy)$ is super Schur $\gamma$-positive if and only if $f(\xx/0;t)$ is Schur  $\gamma$-positive if and only if $f(0/\yy;t)$ is Schur  $\gamma$-positive. This is because the maps from supersymmetric functions to symmetric functions defined by 
$$
s_\lambda(\xx/\yy) \longmapsto s_\lambda(\xx) \quad \mbox{ and }  \quad s_\lambda(\xx/\yy) \longmapsto s_{\lambda'}(\yy) 
$$
are bijective. 
Theorem \ref{gamma-symmetric-h} says that the coefficients of the $\gamma$-Chow polynomials associated to $(e_{i-j}(\xx/\yy))_{i,j=0}^\infty$ are super Schur nonnegative.

\begin{theorem}\label{gamma-symmetric-h}
Let $R=(e_{i-j}(\xx/\yy))_{i,j=0}^\infty$.  Then the polynomials $d_n(t), H_n(t), A_n(t)$ and $G_n(t)$, $n \geq 0$, corresponding to $R$ are all 
super Schur $\gamma$-positive. 

This same is true for the various Chow polynomials in Proposition \ref{trunc-ident} for $R$. 
\end{theorem}

\begin{proof}
By Corollary \ref{cor: coefficients of gamma as determinants}, the coefficients $\gamma_k(\xx/\yy)$ of the $\gamma$-polynomial corresponding to $H_n(t)$ are nonnegative sums of minors of $(e_{i-j}(\xx/\yy))_{i,j=0}^\infty$. The same is true for $\gamma$-polynomial corresponding to $d_n(t),  A_n(t)$ and $G_n(t)$ by \eqref{dntruncR}, \eqref{aug-matrx} and \eqref{DR-def}. Each such minor is equal to a supersymmetric skew Schur function by the Jacobi-Trudy identity. 
Moreover, 
$$
s_{\lambda/\mu}(\xx/\yy) = \sum_{\nu} c_{\mu \nu}^\lambda \cdot s_{\nu}(\xx/\yy),
$$
where $c_{\mu \nu}^\lambda\geq 0$ is a Littlewood-Richardson coefficient, see e.g. \cite[Chapter 2]{Moens}. The theorem follows. 
\end{proof}

\bibliographystyle{amsalpha}
\bibliography{bibliography}

\end{document}